\documentclass[11pt,a4paper,leqno]{article}
\usepackage[utf8]{inputenc}
\usepackage{amsmath}
\usepackage{amsfonts}
\usepackage{amssymb}
\usepackage{amsthm}
\usepackage{makeidx}
\usepackage{graphicx}
\usepackage{authblk}
\usepackage{subfig}
\usepackage{subeqnarray}
\usepackage[T1]{fontenc}
\usepackage[top=2cm, bottom=2cm, left=2cm, right=2cm]{geometry}
\usepackage{marginnote}
\usepackage{comment}

\usepackage{imakeidx}
\usepackage{hyperref}
\usepackage{color}
\usepackage{mathtools}
\mathtoolsset{showonlyrefs}
\usepackage{tabularx}
\newcommand{\unf}[1]{\mathcal{T}_#1}
\newcommand{\averaging}{\mathcal{U}_\varepsilon}

\newcommand{\eps}{\varepsilon}
\newcommand{\eto}{\stackrel{\eps\to 0}{\longrightarrow}}

\newcommand{\weto}{\stackrel{\eps\to 0}{\rightharpoonup}}

\usepackage{imakeidx}
\usepackage{hyperref}
\usepackage{color}
\makeindex[intoc]
\title{Homogenization in 3D thin domains  \\ with oscillating boundaries of different orders}
\date{2024}

\author[1,2]{Jos\'e M. Arrieta\thanks{e-mail: arrieta@mat.ucm.es.}}

\author[3,4]{Jean Carlos Nakasato\thanks{e-mail: j.c.nakasato@math.sci.hokudai.ac.jp or nakasato@ime.usp.br}}

\author[5]{Manuel Villanueva-Pesqueira\thanks{e-mail: mvillanueva@comillas.edu}}

\affil[1]{Dept. An\'alisis Mat y Matem\'atica Aplicada \\  Universidad Complutense de Madrid, Spain}
\affil[2]{ Instituto de Ciencias Matem\'aticas
CSIC-UAM-UC3M-UCM \\ C/Nicol\'as Cabrera 13-15, Cantoblanco, 28049 Madrid, Spain }

\affil[3]{Department of Mathematics, Faculty of Science, University of Hokkaido, Japan}
\affil[4]{Departamento de Matem\'atica Aplicada, Instituto de Matem\'atica e Estat\'istica, Universidade de S\~ao Paulo, Brazil}

\affil[5]{Departamento de  Matem\'atica Aplicada, Universidad Pontificia Comillas, Spain}

\newtheorem{theorem}{Theorem}[section]

\newtheorem{definition}[theorem]{Definition}
\newtheorem{proposition}[theorem]{Proposition}

\numberwithin{equation}{section}
\begin{document}
	\maketitle
	\begin{abstract}
        This paper presents an extension of the unfolding operator technique,  initially applied to two-dimensional domains, to the realm of three-dimensional thin domains. The advancement of this methodology is pivotal, as it enhances our understanding and analysis of three-dimensional geometries, which are crucial in various practical fields such as engineering and physics. Our work delves into the asymptotic behavior of solutions to a reaction-diffusion equation with Neumann boundary conditions set within such a oscillatory 3-dimensional thin domain. The method introduced enables the deduction of effective problems across all scenarios, tackling the intrinsic complexity of these domains. This complexity is especially pronounced due to the possibility of diverse types of oscillations occurring along their boundaries.
	\end{abstract}

	\noindent \emph{Keywords:} Reaction-diffusion equations, Neumann boundary condition, Thin domains, Oscillatory boundary, Homogenization. \\
	\noindent 2010 \emph{Mathematics Subject Classification.} 35B25, 35B40, 35J92.

	\section{Introduction}

    Thin domains with oscillating boundaries have garnered significant interest in the research community due to their natural appearance in modeling real-world phenomena. Most of the applications in real-life problems posses boundaries that are not perfectly smooth, presenting a lot of irregularities that affects significantly the effective behavior of the considered model, which, in general, involves a Partial Differential Equation (PDE) posed in thin domains with rough boundary. As researchers could see, such distortions on the boundary may have significantly influence on the effective behavior of the considered PDE. This has been motivating many investigations to develop and employ asymptotic analysis techniques to determine the effective behavior on a lower-dimensional domain. 
    
    For instance, thin structures with oscillating boundaries are prevalent in various scientific fields, such as fluid dynamics (lubrication), solid mechanics (thin rods, plates or shells) and even physiology (blood circulation), \cite{Hamrock,Smith,Tabeling}. Refer to \cite{AnBra, BouCiu,BGG07B} for specific examples of applied problems. 
 
     In this work, we extend the unfolding operator introduced in \cite{AM2} from two-dimensional to three-dimensional thin domains. This extension is significant from both a practical and theoretical standpoint. Practically, three-dimensional domains have substantial applications across various fields such as engineering and physics, making their comprehension vital for the progress in these disciplines. Theoretically, the analysis of these domains introduces additional complexity due to the potential for combining different types of oscillations at their boundaries. This enriches and complicates their study, presenting deeper and more varied mathematical challenges
	
	Throughout this paper, we consider thin domains with different kind of oscillations at the top  boundary, given by
	\begin{equation*}
	R^\varepsilon=\left\lbrace (x_1,x_2,x_3)\in \mathbb{R}^3:(x_1,x_2)\in \omega,\quad 0<x_3<\eps^\gamma g\left(\dfrac{x_1}{\varepsilon^\alpha},\dfrac{x_2}{\varepsilon^\beta}\right)\right\rbrace,\,\,\,\,0<\varepsilon\ll 1,\, 
	\end{equation*}
	where $0\leq \alpha< \beta,\,0<\gamma$, $\omega\subset\mathbb{R}^2$ is an open, bounded, connected and regular set and $g$ is a bounded periodic function not necessarily smooth. For an example of oscillatory boundary, see Figure \ref{fig21}.
    
    Without loss of generality, we can consider $\gamma=1$ since, if this is not the case, the thickness of the domain can be redefined as $\varepsilon=\eps^\gamma$. Therefore, the thin domains considered are given by
    \begin{equation}\label{TDs}\tag{D}
	R^\varepsilon=\left\lbrace (x_1,x_2,x_3)\in \mathbb{R}^3:(x_1,x_2)\in \omega,\quad 0<x_3<\varepsilon g\left(\dfrac{x_1}{\varepsilon^\alpha},\dfrac{x_2}{\varepsilon^\beta}\right)\right\rbrace,\,\,\,\,0<\varepsilon\ll 1,\, 
	\end{equation}

	Note that various values of $\alpha$ and $\beta$ greater than zero will result in distinct kinds of oscillatory patterns or roughness at the boundary. Much more complex behaviors appear than in two dimensions. In fact, it is possible to have oscillations of one type in the $x_1$ -direction and oscillations of another type in $x_2$. Even in the simplest case, $\alpha$ or $\beta$ is zero, a locally periodic behavior is obtained.
	
	Considering that the focus of this article lies in elaborating the technique of adapting the unfolding operator to a three-dimensional context, rather than on the specific problems it solves, we have chosen to illustrate this approach through a particular elliptic problem. This choice underscores our understanding that, while the method holds potential for broader application across more intricate equations, our emphasis remains on the nuances of the geometrical adaptation. Then, in the sequel, we address the following elliptic boundary-value problem.
	\begin{equation}\label{problem}\tag{P}
	\left\lbrace\begin{array}{lll}
	-\Delta u^\varepsilon +u^\varepsilon=f^\varepsilon\quad\textrm{in}\quad R^\varepsilon,\\
	\dfrac{\partial u^\varepsilon}{\partial \nu^\varepsilon}=0\quad \textrm{on}\quad\partial R^\varepsilon,
	\end{array}\right.
	\end{equation}
	where $\nu^\varepsilon$ is the normal unit outward vector to $\partial R^\varepsilon$, $f^\varepsilon\in L^{2}(R^\varepsilon)$.
	
	The variational formulation of \eqref{problem} is
	\begin{equation}\label{variationalproblem}
	\int_{R^\varepsilon}\left(\nabla u^\varepsilon\nabla\varphi +u^\varepsilon\varphi\right) dx=\int_{R^\varepsilon} f^\varepsilon\varphi dx
	\end{equation}
	for all $\varphi\in H^1(R^\varepsilon)$.
	
	Lax-Milgram Theorem ensures the existence and uniqueness of solutions for problem \eqref{variationalproblem} for any fixed $\varepsilon>0$. It is important to point out that the solutions' behavior is primarily determined by the value of the parameters $\alpha$ and $\beta$ respect to the thickness of the domain. Additionally, given that the thickness of the domain $R^\varepsilon$ is of order $\eps$, it is anticipated that the sequence of solutions $u_\varepsilon$ will converge to a function with $n-1$ variables as $\varepsilon$ approaches zero.
	
	Thus, the aim of this article is to present the unfolding method as a general approach that simplifies the process of obtaining the homogenized limit problem  for problem \eqref{problem} considering all possible combinations of the values of $\alpha$ and $\beta$. Additionally, just like in the two-dimensional case, non-smooth periodic oscillatory boundaries may be considered.

	\begin{figure}[!htb]
		\centering{
			\scalebox{0.35}{\includegraphics{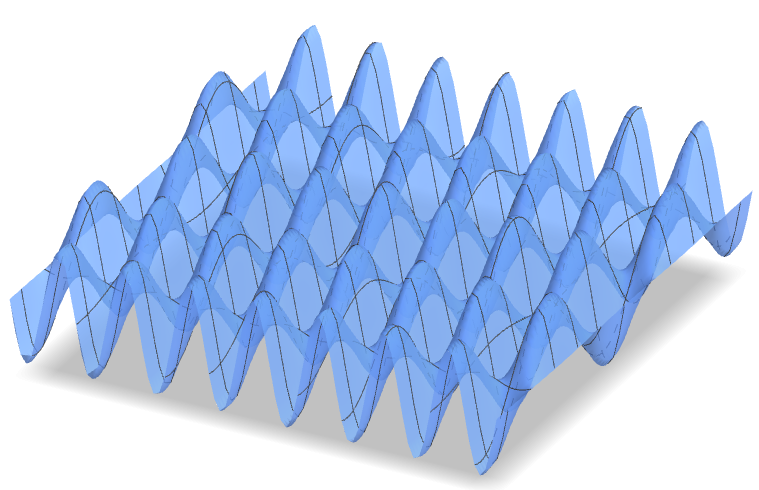}}
		}
		\caption{3d rough thin domain}
		\label{fig21}
	\end{figure}

    As previously mentioned, depending on the value of the parameters $\alpha$ and $\beta$, we obtain distinct lower dimensional limit problems taking the general form
    \begin{equation}
        \left\lbrace
            \begin{array}{c}
                 -(q_1u_{x_1})_{x_1}-(q_2u_{x_2})_{x_2}+u=\Bar{f}\quad\textrm{in}\quad \omega, \\
                  (q_1u_{x_1},q_2u_{x_2})\eta =0 \quad\textrm{on}\quad\partial\omega,
            \end{array}
        \right.
    \end{equation}  
    where $\eta$ denotes the unit outward normal vector to the boundary $\partial\omega$ and the homogenized coefficients $q_i$, $i=1,2$, varies as accordingly to the values of $\alpha$, and $\beta$. The function $\bar f$ is related to the function $f^\eps$.  We left the description of the coefficients of the different limitting problems for the next sections.

    The key ingredient of the current work is the adaptation of the unfolding method introduced in \cite{CiDaGri} for a higher dimensional thin domain with rough boundary. First, the unfolding method was developed for problems with oscillating coefficients and perforated domains (see \cite{Cioranescu}) and then adapted for thin domains with oscillating boundary (see \cite{AM,AM2,JM}). Besides defining the unfolding operator, which allows us to work within a fixed framework, one must address the challenge, common to all homogenization problems, of determining the limit of the partial derivatives. To achieve this, our primary innovation lies in constructing suitable operators that facilitate the determination of the limit of the unfolded gradients. Unlike in previous works, such as \cite{Cioranescu,AM2}, we encounter a unique complication: it is not possible to identify a single limit function whose gradient matches the limit of the unfolded gradients, owing to the differing scales in the $x_1$ and $x_2$ directions. Consequently, these differing scales lead us to identify various limit functions within distinct Lebesgue-Bochner spaces. The emergence of these diverse spaces necessitates a slightly modified corrector equation for the critical oscillations (i.e., of order $\varepsilon$) and also compels us to employ different test functions for each scenario, due to the interaction between the oscillation orders in each direction. Furthermore, it is pertinent to highlight that our ideas and techniques can be easily adapted to scenarios involving higher dimensions, though it must be noted that in such cases, the notation might become excessively cumbersome.

    Now, we perform a brief overview on the literature.  For pioneering works, we mention \cite{AnBra,BC,HaleRaugel} where thin domains with or without oscillatory boundaries were studied. In \cite{AnBra}, thin domains with oscillatory boundary were considered in context of $\Gamma$-convergence and in \cite{BC}, where the Stokes system were studied. In \cite{HaleRaugel}, the authors studied a parabolic equation and its asymptotic dynamics in a standard thin domain (i.e. without oscillations). 
    
    More recently, the works \cite{arrieta2011,AJM,AM,AM2,MPov,JIM,JM} studied elliptic and parabolic problems on thin domains with rough boundary. These works used several techniques from the classical extension operator or asymptotic expansions to the most recent ones, using the unfolding operator method. 
    It is important to emphasize that in \cite{AM,AM2}, the pioneering works with respect to the unfolding method in oscillating thin domains, a profound study of the method was performed for two-dimensional domains.

\section{Notations and Preliminary Results}
	To study the convergence of the solutions of \eqref{problem}, we fix some notations and recall results concerning the method of unfolding operator which will be essential trough the paper for our analysis.
	
	We consider three-dimensional thin domains defined by \eqref{TDs}. Observe that this domain have an oscillatory behavior at its top boundary. The parameters $\varepsilon$, $\alpha$ and $\beta$  are positive and the function $g$ is a periodic function.
Firstly, we would like to clarify that, for the sake of simplicity in notation we say that a function in $\mathbb{R}^2$ is $L$-periodic if there exist two periods, $L_1$ and $L_2$, such that the function is $L_1$-periodic in the first variable and $L_2$-periodic in the second. 
%Throughout the article, $L$ denotes the rectangle $L=[0,L_1]\times[0, L_2]$ describing the periodic behavior of the upper boundary.

		% \begin{equation}\tag{$H_g$}\label{hyp_Hg}
		% \begin{gathered}
		% g : \mathbb{R}^2\rightarrow \mathbb{R}\textrm{ is a strictly positive, bounded, Lipschitz, }L\textrm{-periodic function.}\qquad\qquad\qquad\qquad\\
		% \textrm{ Moreover, we define}\qquad\qquad\qquad\qquad\qquad\qquad\qquad\qquad\qquad\qquad\qquad\qquad\qquad\qquad\qquad\qquad\\[10pt]
		% g_0 = \min_{(x_1,x_2) \in \mathbb{R}^2} g(x_1,x_2) \quad  \textrm{ and } \quad g_1 = \max_{(x_1,x_2) \in \mathbb{R}^2} g(x_1,x_2)\\[10pt]
		% \textrm{so that }0<g_0\leq g(x_1,x_2)\leq g_1 \textrm{ for all }(x_1,x_2)\in\mathbb{R}^2.\qquad\qquad\qquad\qquad\qquad\qquad\qquad\qquad\qquad
		% \end{gathered}
		% \end{equation}

 The function $g$ which defines the oscillatory boundary satisfies the following hypothesis:
	
 \noindent($\mathbf{H_g}$) {\sl $g : \mathbb{R}^2\rightarrow \mathbb{R}$ is a strictly positive, bounded, lower semicontinuous, $L$-periodic function.  Moreover, we define $$g_0 = \min_{(x_1,x_2) \in \mathbb{R}^2} g(x_1,x_2) \quad  \textrm{ and } \quad g_1 = \max_{(x_1,x_2) \in \mathbb{R}^2} g(x_1,x_2)$$ 
		so that $0<g_0\leq g(x_1,x_2)\leq g_1$ for all $(x_1,x_2)\in\mathbb{R}^2$}.
	%that is, there exist positive constants $g_0$ and $g_1$ such that $$0 < g_0\leq g(x)\leq g_1, \quad \textrm{ for all } x\in \mathbb{R},$$
%	with 
%	$$g_0 = \min_{x \in \mathbb{R}} g(x) \quad  \textrm{ and } \quad g_1 = \max_{x \in \mathbb{R}} g(x).$$ 

	Recall that lower semicontinuous means that $g(x_1^0,x_2^0)\leq \displaystyle\liminf_{(x_1,x_2)\rightarrow (x_1^0,x_2^0)}g(x_1,x_2)$, $\forall (x_1^0,x_2^0)\in\mathbb{R}^2$. 
  
  We will now establish the typical notation for the unfolding operator, adapted specifically to the case of a three-dimensional oscillating thin domain. This tailored notation is essential in ensuring clarity and precision in our definition of the unfolding operator, providing researchers with a standardized framework.

  Let us consider a rectangular grid in $\mathbb{R}^2$ taking into account the periodicity of the problem, each rectangle $\omega_{i,j}$ has a width of $L_1$ and a height of $L_2$: 
  $$ \omega_{i,j}=[iL_1,(i+1)L_1)\times[jL_2,(j+1)L_2),\quad (i,j)\in \mathbb{Z}^2.$$
  
  By analogy with the one dimensional case, for each $(x_1,x_2) \in \mathbb{R}^2$, $[x_1,x_2]_L$ denotes the bottom left vertex of the rectangle where the point is. For instance, if $(x_1,x_2) \in \omega_{i,j}$ we have $[x_1,x_2]_L=(iL_1,jL_2)$. We also define $\{x_1,x_2\}_L=(x_1,x_2)-[(x_1,x_2)]_L$.

  \begin{figure}[!htb]
		\centering{
			\scalebox{0.5}{\includegraphics{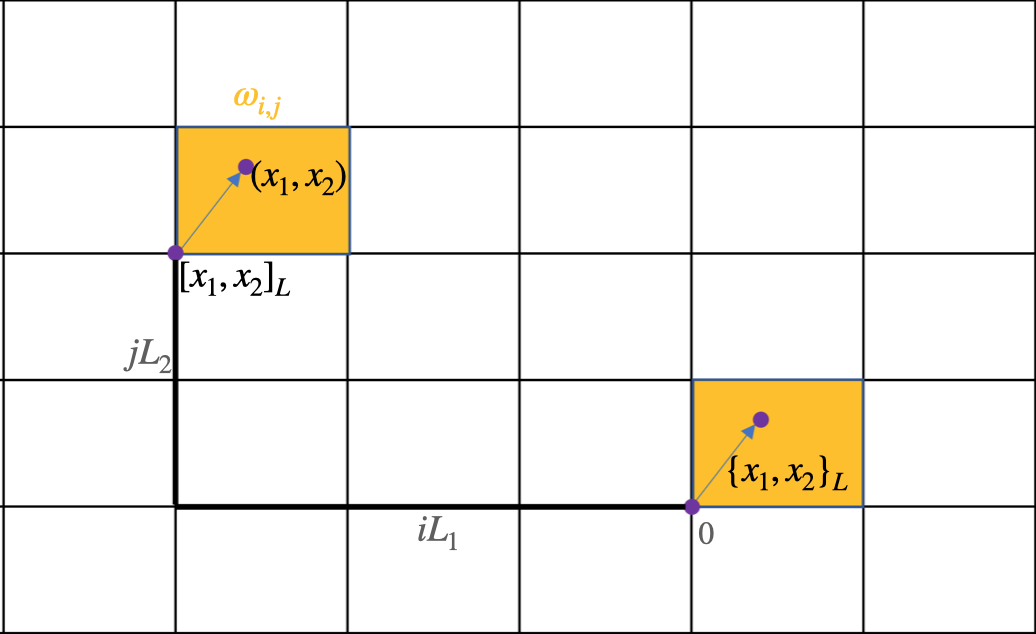}}
		}
		\caption{2d grid from the periods}
	\end{figure}

  In particular, for each $\eps>0$ we can write:
$$\left(\frac{x_1}{\varepsilon^\alpha},\frac{x_2}{\varepsilon^\beta}\right)=\Big[\frac{x_1}{\varepsilon^\alpha},\frac{x_2}{\varepsilon^\beta}\Big]_L+\Big\{\frac{x_1}{\varepsilon^\alpha},\frac{x_2}{\varepsilon^\beta}\Big\}_L.$$

Notice that,  if \(\alpha=\beta\), as in classical periodic homogenization, it is straightforward to obtain the relationship between the macro and micro scales. In fact, the rescaled rectangles $\eps^\alpha \omega_{i,j}$ represent the microscopic scale and we have
$$(x_1,x_2)=\eps^\alpha\Big(\Big[\frac{x_1}{\varepsilon^\alpha},\frac{x_2}{\varepsilon^\alpha}\Big]_L+\Big\{\frac{x_1}{\varepsilon^\alpha},\frac{x_2}{\varepsilon^\alpha}\Big\}_L\Big).$$

However, when $\alpha \neq \beta$,  the grid representing the microscopic scale is not formed by rectangles that are homothetic to $\omega_{i,j}$. In fact, we define the following family of rectangles
$$ \omega^\varepsilon_{i,j}=[iL_1\varepsilon^\alpha,(i+1)L_1\varepsilon^\alpha)\times[jL_2\varepsilon^\beta,(j+1)L_2\varepsilon^\beta),\quad (i,j)\in \mathbb{Z}^2$$
which are obtained  by shrinking by a factor $\eps^\alpha$ in the $x-$direction and by a factor $\eps^\beta$ in the $y-$direction. Consequently, we will use the following notation 
\begin{equation}
	\left(x_1,x_2\right)=\left(\varepsilon^\alpha\Big(\left[\frac{x_1}{\varepsilon^\alpha}\right]_{L_1}L_1+\left\lbrace\frac{x_1}{\varepsilon^\alpha}\right\rbrace_{L_1}\Big),\varepsilon^\beta\Big(\left[\frac{x_2}{\varepsilon^\beta}\right]_{L_2}L_2+\left\lbrace\frac{x_2}{\varepsilon^\beta}\right\rbrace_{L_2}\Big)\right), \; \forall (x_1,x_2)\in \mathbb{R}^2.
	\end{equation}
Notice that, according to the previous notation we have
$$\Big[\frac{x_1}{\varepsilon^\alpha},\frac{x_2}{\varepsilon^\beta}\Big]_L=\Big(\left[\frac{x_1}{\varepsilon^\alpha}\right]_{L_1}L_1,\left[\frac{x_2}{\varepsilon^\beta}\right]_{L_2}L_2\Big) \hbox{ and } \Big\{\frac{x_1}{\varepsilon^\alpha},\frac{x_2}{\varepsilon^\beta}\Big\}_L=\Big(\left\lbrace\frac{x_1}{\varepsilon^\alpha}\right\rbrace_{L_1},\left\lbrace\frac{x_2}{\varepsilon^\beta}\right\rbrace_{L_2}\Big).$$

Furthermore, given any domain $\omega \in \mathbb{R}^2$, we will distinguish between the rectangles $\omega^\varepsilon_{i,j}$ that are completely contained within $\omega$
and those that are not. Then, we denote by $\omega_\varepsilon$ and $\Lambda_\varepsilon$ the sets
%$\overline{\omega^\varepsilon_{i,j}}$ included in $\omega$:
	$$
	\omega_\varepsilon=\bigcup_{(i,j)\in S_\varepsilon}\overline{\omega^\varepsilon_{i,j}}\quad\textrm{and}\quad \Lambda_\varepsilon=\omega\backslash\omega_\varepsilon,
	$$
where $S_\varepsilon\subset\mathbb{Z}^2 $ is such that for each $(i,j)\in S_\varepsilon$, we have 
	$\overline{\omega^\varepsilon_{i,j}}\subset\omega$. 

Notice that $\omega_\eps\subset\omega$ and $\Lambda_\varepsilon \eto 0$ in the sense of the measure. Both sets will play an important role in defining the unfolding operator, see Figure \ref{fig21}. 

\begin{figure}[!htb]
		\centering{
			\scalebox{0.45}{\includegraphics{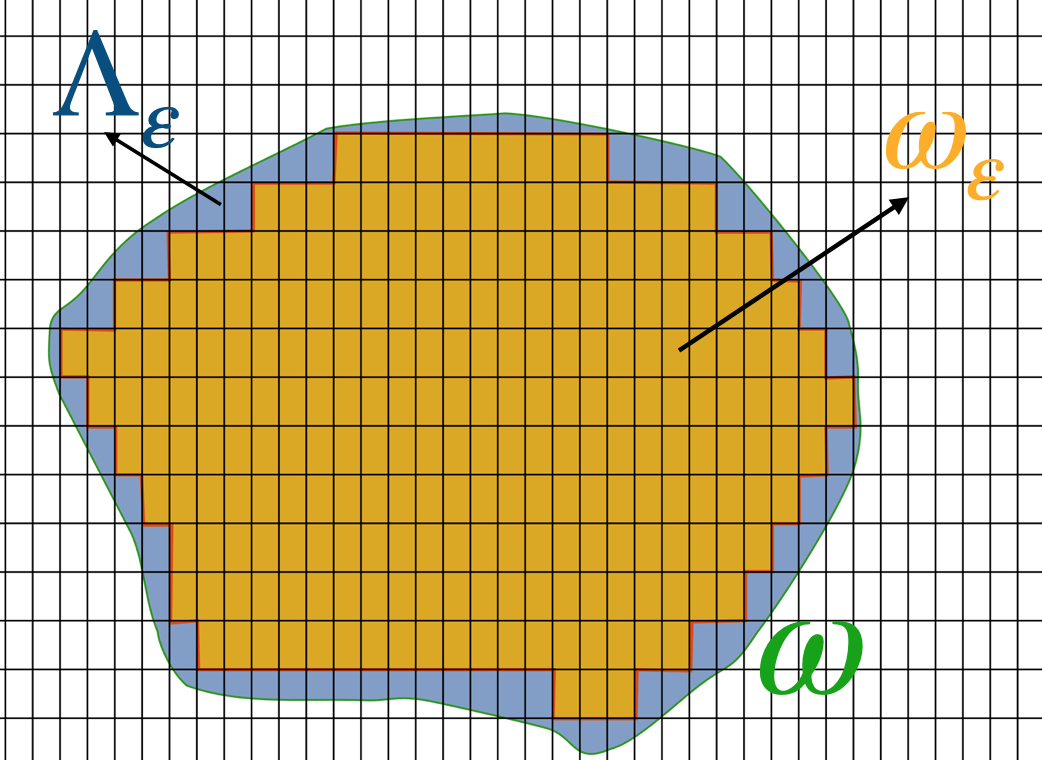}}
		}
		\caption{Sets $\omega$, $\omega_\eps$ and $\Lambda_\eps$}
	\end{figure}
Then, we can split the thin domain $R_\eps$ into two parts:
	$$
	R^\varepsilon_0=\left\{(x_1,x_2,x_3)\in\mathbb{R}^3:(x_1,x_2)\in\omega_\varepsilon,\, 0<x_3<\eps g\left(\dfrac{x_1}{\varepsilon^\alpha},\dfrac{x_2}{\varepsilon^\beta}\right) \right\}
	$$
	and
	$$
	R^\varepsilon_1=\left\{(x_1,x_2,x_3)\in\mathbb{R}^3:(x_1,x_2)\in\Lambda_\varepsilon,\, 0<x_3<\eps g\left(\dfrac{x_1}{\varepsilon^\alpha},\dfrac{x_2}{\varepsilon^\beta}\right) \right\}.
	$$

 We reserve the notation $Y^*$ for the reference cell which describes the oscillating domain
$$
	Y^*=\left\{(y_1,y_2,y_3)\in\mathbb{R}^3:0<y_1<L_1,\,0<y_2<L_2,\,0<y_3<g(y_1,y_2) \right\}.
	$$
Moreover,  since we have two different scales of oscillation for the variables \(x\) or \(y\), the following notation will be used:
\begin{align*}
	Y^*(y_1)=\left\{(y_2,y_3)\in\mathbb{R}^2:0<y_2<L_2,\,0<y_3<g(y_1,y_2) \right\}, \quad y_1 \in [0,L_1],\\
	Y^*(y_2)=\left\{(y_1,y_3)\in\mathbb{R}^2:0<y_1<L_1,\,0<y_3<g(y_1,y_2) \right\}, \quad y_2 \in [0, L_2],\\
    Y^*(y_1,y_2)=\left\{y_3\in\mathbb{R}:0<y_3<g(y_1,y_2) \right\}, \quad (y_1, y_2) \in L.
 \end{align*}

To enhance readability, we will represent \(x\) and \(y\) as vectors \(x = (x_1, x_2, x_3)\) and \(y = (y_1, y_2, y_3)\), respectively.

 Finally, recall some very commonly used notations in homogenization.  The subindex $\sharp$ denotes periodicity. For instance, $\mathcal{C}^\infty_{\sharp i}(\omega\times Y^*)$ consists of all functions $\varphi$ wich are obtained as restrictions to $\omega\times Y^*$ of functions in $\mathbb{R}^5$ which are $L_i-$periodic in the $y_i-$variable for $i=1,2,3$. For a measurable set $A\subset\mathbb{R}^n$, denote the average of a function $\varphi$ in $A$ as $\langle\varphi\rangle_{A}=\frac{1}{|A|}\int_A\varphi$.

	\subsection{The unfolding operator}
      In this section, we extend the definition of the unfolding operator that was originally given for 2-dimensional thin domains in \cite{AM2}. For the n-dimensional case, which exhibits the same type of oscillations in all directions, refer to \cite{JM}. Additionally, we present its main properties.

	\begin{definition}
		Let $\varphi$ be a Lebesgue-measurable function defined in $R^\varepsilon$. The unfolding operator $\mathcal{T}_\varepsilon$ acting on $\varphi$ is defined as the following function defined in $\omega\times Y^*$
		\begin{eqnarray*}
              \mathcal{T}_\varepsilon(\varphi)(x_1,x_2,y)=\left\lbrace\begin{array}{ll}
				\varphi\left(\varepsilon^\alpha \left[\frac{x_1}{\varepsilon^\alpha}\right]_{L_1}L_1+\varepsilon^\alpha y_1,\varepsilon^\beta\left[\frac{x_2}{\varepsilon^\beta}\right]_{L_2}L_2+\varepsilon^\beta y_2,\eps y_3\right),\,&(x_1,x_2,y)\in \omega_\varepsilon \times Y^*,\\ \\
				0, & (x_1,x_2,y)\in \Lambda_\varepsilon\times Y^*.
			\end{array}\right.
		\end{eqnarray*}
		
	\end{definition}
	
       Next, we establish some basic properties of $\unf{\varepsilon}$ that will play an essential role in the paper. These properties do not depend on the values of the parameters $\alpha$ and $\beta$.

	\begin{proposition}\label{Proposition_unfolding_properties}
 The unfolding operator has the following properties:
		\begin{enumerate}
			\item[(a)] $\unf{\varepsilon}$ is linear with respect to $+$ and $\cdot$ operations.
			\item[(b)] Let $\varphi$ be a Lebesgue function defined in $Y^*$ which is extended periodically in $(y_1,y_2)$. Then, $\varphi^\varepsilon(x_1,x_2,x_3)=\varphi\left(\frac{x_1}{\varepsilon^\alpha},\frac{x_2}{\varepsilon^\beta},\frac{x_3}{\eps}\right)$ is mesurable in $R^\varepsilon$ and
			$
			\mathcal{T}_\varepsilon(\varphi^\varepsilon)(x_1,x_2,y_1,y_2,y_3)=\varphi(y_1,y_2,y_3).
			$
			Moreover, if $\varphi\in L^2(Y^*)$, then $\varphi^\varepsilon\in L^2(R^\varepsilon)$;
			\item[(c)] For all $\varphi \in L^1(R^\varepsilon)$, we have
			$$\frac{1}{L_1L_2}\int_{\omega \times Y^*}\mathcal{T}_\varepsilon(\varphi)(x_1, x_2,y)dx_1dx_2dy=\frac{1}{\eps}\int_{R^\varepsilon_0}\varphi(x_1,x_2,x_3)dx;$$
   %-\frac{1}{\varepsilon}\int_{R_1^\varepsilon}\varphi(x_1,x_2,x_3)dx;$$
			\item[(d)]$T_\varepsilon(\varphi)\in L^p\left(\omega \times Y^*\right)$ for all $\varphi \in L^p(R^\varepsilon)$ with
			\begin{equation*}
			\left|\left|\mathcal{T}_\varepsilon(\varphi)\right|\right|_{L^p\left(\omega \times Y^*\right)}=\left(\frac{L_1L_2}{\eps}\right)^{\frac{1}{p}}\left|\left|\varphi\right|\right|_{L^p(R_0^\varepsilon)}\leq \left(\frac{L_1L_2}{\eps}\right)^{\frac{1}{p}}\left|\left|\varphi\right|\right|_{L^p(R^\varepsilon)}.
			\end{equation*}
			\item[(e)] $
			\partial_{y_1}\unf{\varepsilon}(\varphi)=\varepsilon^\alpha \unf{\varepsilon}(\partial_{x_1}\varphi)\mbox{, }\partial_{y_2}\unf{\varepsilon}(\varphi)=\varepsilon^\beta\unf{\varepsilon}(\partial_{x_2}\varphi)\textrm{ and } \partial_{y_3}\unf{\varepsilon}(\varphi)=\eps \unf{\varepsilon}(\partial_{x_3}\varphi)\mbox{ a.e. } \omega \times Y^*$ for all $\varphi \in W^{1,p}(R^\varepsilon)$.
			\item[(f)] If $\varphi \in W^{1,p}(R^\varepsilon)$, then $\mathcal{T}_\varepsilon(\varphi)  \in L^p\left(\omega ;W^{1,p}(Y^*)\right)$ with
			$$
			\begin{gathered}
			\left|\left|\partial_{y_1}\mathcal{T}_\varepsilon(\varphi)\right|\right|_{L^p\left(\omega \times Y^*\right)}=\varepsilon^\alpha \left(\frac{L_1L_2}{\eps}\right)^{\frac{1}{p}}\left|\left|\partial_{x_1}\varphi\right|\right|_{L^p(R_0^\varepsilon)}\leq \varepsilon^\alpha\left(\frac{L_1L_2}{\eps}\right)^{\frac{1}{p}}\left|\left|\partial_{x_1}\varphi\right|\right|_{L^p(R^\varepsilon)},\\
			\left|\left|\partial_{y_2}\mathcal{T}_\varepsilon(\varphi)\right|\right|_{L^p\left(\omega \times Y^*\right)}=\varepsilon^\beta\left(\frac{L_1L_2}{\eps}\right)^{\frac{1}{p}}\left|\left|\partial_{x_2}\varphi\right|\right|_{L^p(R_0^\varepsilon)}\leq \varepsilon^\beta\left(\frac{L_1L_2}{\eps}\right)^{\frac{1}{p}}\left|\left|\partial_{x_2}\varphi\right|\right|_{L^p(R^\varepsilon)},\\
			\left|\left|\partial_{y_3}\mathcal{T}_\varepsilon(\varphi)\right|\right|_{L^p\left(\omega \times Y^*\right)}=\eps\left(\frac{L_1L_2}{\eps}\right)^{\frac{1}{p}}\left|\left|\partial_{x_3}\varphi\right|\right|_{L^p(R_0^\varepsilon)}\leq \eps\left(\frac{L_1L_2}{\eps}\right)^{\frac{1}{p}}\left|\left|\partial_{x_3}\varphi\right|\right|_{L^p(R^\varepsilon)}.
			\end{gathered}
			$$
		\end{enumerate}
	\end{proposition}
	\begin{proof}
		The proof is very similar to \cite{AM,AM2}. Then, we just prove property c), which gives the relation between the integrals in the thin domain and integrals in the new fixed domain in $\mathbb{R}^5:$
\begin{align*}
& \frac{1}{L_1L_2}\int_{ \omega \times Y^*} \mathcal{T_\eps(\varphi)} (x_1,x_2,y) \, dx_1 dx_2 dy\\ & = \frac{1}{L_1L_2}\int_{ \omega^\eps \times Y^*}\varphi\left(\varepsilon^\alpha \left[\frac{x_1}{\varepsilon^\alpha}\right]_{L_1}L_1+\varepsilon^\alpha y_1,\varepsilon^\beta\left[\frac{x_2}{\varepsilon^\beta}\right]_{L_2}L_2+\varepsilon^\beta y_2,\eps y_3\right)dx_1 dx_2 dy\\
&= \frac{1}{L_1L_2}\sum_{(i,j) \in S^\eps}\int_{\omega^\eps_{i,j}}\int_{Y^*} \varphi\left(\varepsilon^\alpha \left[\frac{x_1}{\varepsilon^\alpha}\right]_{L_1}L_1+\varepsilon^\alpha y_1,\varepsilon^\beta\left[\frac{x_2}{\varepsilon^\beta}\right]_{L_2}L_2+\varepsilon^\beta y_2,\eps y_3\right)dx_1 dx_2 dy \\
%&= \frac{1}{L_1L_2}\sum_{k=0}^{N_\eps}\int_{k \eps^\beta L}^{(k+1)\eps^\beta L}\int_{Y^*} \varphi \Big( \eps^\beta kL + \eps^\beta y_1, \eps y_2\Big)\, dy_1 dy_2 dx\\
&=\eps^\alpha \eps^\beta \sum_{(i,j) \in S^\eps} \int_{Y^*} \varphi\left(\varepsilon^\alpha iL_1+\varepsilon^\alpha y_1,\varepsilon^\beta j L_2+\varepsilon^\beta y_2,\eps y_3\right)dy \\
&=\frac{1}{\eps}\sum_{(i,j) \in S^\eps} \int_{\omega^\eps_{i,j}}\int_{0}^{\eps g(x_1/\eps^\alpha,x_2/\eps^\beta)} \varphi(x)\, dx\\
&=  \frac{1}{\eps}\int_{R_0^\eps} \varphi (x) dx.
\end{align*}
\end{proof}
	
	From now on, we will use the following notation for the rescaled norms in the thin domain $R^\varepsilon$
	\begin{eqnarray*}
		&&\left|\left|\left|\varphi\right|\right|\right|_{L^p(R^\varepsilon)}=\Big(\frac{1}{\eps}\Big)^{1/p}\left|\left|\varphi\right|\right|_{L^p(R^\varepsilon)},\forall\varphi\in L^p(R^\varepsilon),\\
		&&\left|\left|\left|\varphi\right|\right|\right|_{W^{1,p}(R^\varepsilon)}=\Big(\frac{1}{\eps}\Big)^{1/p}\left|\left|\varphi\right|\right|_{W^{1,p}(R^\varepsilon)}\forall\varphi\in W^{1,p}(R^\varepsilon).
	\end{eqnarray*}

	Notice that Proposition \ref{Proposition_unfolding_properties} is essential to pass to the limit since allows us to transform an integral over $R^\varepsilon$ into one over the fixed set $\omega \times Y^*$.
	Hence, the unfolding criterion for integrals (u.c.i.) plays an important role.
	
	\begin{definition}
		A sequence $\left(\varphi^\varepsilon\right)$ satisfies the unfolding criterion for integrals (u.c.i) if
		\begin{equation*}
		\frac{1}{\varepsilon}\int_{R_1^\varepsilon}|\varphi^\varepsilon|dx_1dx_2\rightarrow 0.
		\end{equation*}
	\end{definition}
	\begin{comment}
	It is known that any sequence $\left(\varphi^\varepsilon\right) \subset L^p(R^\varepsilon)$ with norm $\left|\left|\left|\cdot\right|\right|\right|_{L^p(R^\varepsilon)}$ uniformly bounded satisfies the (u.c.i).
	Moreover, if we have $\left(\psi^\varepsilon\right)$ set as
	\begin{equation*}
	\psi^\varepsilon(x_1,x_2,x_3)=\psi\left(\frac{x}{\varepsilon},\frac{y}{\varepsilon^\beta},\frac{z}{\varepsilon}\right)
	\end{equation*}
	for any $\psi\in L^{p}(Y^*)$, then $(\varphi^\varepsilon\psi^\varepsilon)$ also satisfies (u.c.i). {\color{red}See Arrieta ... for similar results}
	\end{comment}
	The proofs of the next Propositions are analogous to \cite{AM,AM2}. 
	\begin{proposition}
		Let $\varphi^\varepsilon$ be a sequence in $L^p(R^\varepsilon)$, $1<p\leq\infty$ with $|||\varphi^\varepsilon|||_{L^p(R^\varepsilon)}$ uniformly bounded. Then, $\varphi^\varepsilon$ satisfies u.c.i. Furthermore, 
		\begin{itemize}
			\item if $\psi^\varepsilon\in L^q(R^\varepsilon)$, $\frac{1}{p}+\frac{1}{q}=\frac{1}{r}$, $r>1$, then $\varphi^\varepsilon\psi^\varepsilon$ satisfies the u.c.i.
			\item if $\phi\in L^q(\omega)$, $\frac{1}{p}+\frac{1}{q}=1$, then $\varphi^\varepsilon\phi$ satisfies u.c.i.
			\item If $\psi^\varepsilon(x_1,x_2,x_3)=\psi\left(\frac{x_1}{\varepsilon^\alpha},\frac{x_2}{\varepsilon^\beta},\frac{x_3}{\eps} \right)$, where $\psi\in L^q(Y^*)$, then $\varphi^\varepsilon\psi^\varepsilon$ satisfies u.c.i.
		\end{itemize}
	\end{proposition}
	\begin{proposition}\label{Proposition_convergence_test_functions}
		\begin{enumerate}
			%		\item Let $f\in L^2((0,1);L^2_\#(Y^*))$ and extend it periodically in $y_1$-direction. If we set
			%		$
			%		f^\varepsilon(x_1,x_2):=f\left(x,\frac{x}{\varepsilon},\frac{y}{\varepsilon}\right)
			%		$
			%		then
			%		$\mathcal{T}_\varepsilon f^\varepsilon\rightarrow f \mbox{ strongly in }L^2\left((0,1)\times Y^*\right).$
			
			\item Let $\varphi\in L^p(\omega )$. Then,
			$
			\mathcal{T}_\varepsilon\varphi\rightarrow \varphi\mbox{ strongly in }L^p\left(\omega \times Y^*\right).
			$
			
			\item Let $(\varphi^\varepsilon)$ be a sequence in $L^p(\omega )$ such that
			$
			\varphi^\varepsilon\rightarrow \varphi\mbox{ strongly in }L^p(\omega ).
			$
			Then,
			$$\mathcal{T}_\varepsilon\varphi^\varepsilon\rightarrow \varphi\mbox{ strongly in }L^p\left(\omega \times Y^*\right).$$
		\end{enumerate}
	\end{proposition}

	Next, we recall a convergence result which does not depend on the value of the parameter $\alpha$ or $\beta$. To do that, we first introduce a suitable decomposition to functions $\varphi\in W^{1,p}(R^\varepsilon)$ where the geometry of the thin domains plays a crucial role. 
	We write $\varphi(x_1,x_2,x_3)=V(x_1,x_2)+\varphi_r(x_1,x_2,x_3)$ where $V$ is defined as follows
	\begin{equation}\label{Vauxiliar}
	V(x_1,x_2):=\frac{1}{\eps g_0}\int_0^{\eps g_0}\varphi(x_1,x_2,x_3) \, dx_3 \quad \mbox{ a.e. } (x_1,x_2)\in \omega.
	\end{equation}
	\begin{proposition}
		\label{propositionconvergence}
		Let $(\varphi^\varepsilon)\subset W^{1,p}(R^\varepsilon)$, $1<p<\infty$, with $\left|\left|\left|\varphi^\varepsilon\right|\right|\right|_{W^{1,p}(R^\varepsilon)}$ uniformly bounded and $V^\varepsilon(x_1,x_2)$ defined as in \eqref{Vauxiliar} for each $\varphi_{\eps}$. Then, there exists a function $\varphi\in W^{1,p}(\omega )$ such that, up to subsequences
		\begin{eqnarray*}
                && V^\eps \weto \varphi \quad \hbox{ weakly in } W^{1,p}(\omega)\quad\textrm{and strongly in}\quad L^p(\omega),\\
			&&\left|\left|\left|\varphi^\varepsilon-\varphi\right|\right|\right|_{L^{p}(R^\varepsilon)}\rightarrow 0,\\
			&&\mathcal{T}_\varepsilon(\varphi^\varepsilon) \to  \varphi\mbox{ strongly in } L^p\left(\omega ;W^{1,p}(Y^*)\right).
		\end{eqnarray*}
	\end{proposition}
	\begin{proof}
		Let us prove that $V^\varepsilon$ is uniformly bounded:
		\begin{align*}
		\|V^\varepsilon\|_{L^p(\omega )}^p= \int_{\omega }\left|\dfrac{1}{\eps g_0}\int_0^{\eps g_0}\varphi^\varepsilon(x_1,x_2,x_3)dx_3\right|^p dx_1dx_2\leq c|||\varphi^\varepsilon|||_{L^p(R^\varepsilon)}^p\leq C,\\
		\|\partial_{x_1} V^\varepsilon\|_{L^p(\omega )}^p= \int_{\omega }\left|\dfrac{1}{\eps g_0}\int_0^{\eps g_0}\partial_{x_1}\varphi^\varepsilon(x_1,x_2,x_3)dx_3\right|^p dx_1dx_2\leq c|||\partial_{x_1}\varphi^\varepsilon|||_{L^p(R^\varepsilon)}^p\leq C,\\
		\|\partial_{x_2} V^\varepsilon\|_{L^p(\omega )}^p= \int_{\omega }\left|\dfrac{1}{\eps g_0}\int_0^{\eps g_0}\partial_{x_2}\varphi^\varepsilon(x_1,x_2,x_3)dx_3\right|^p dx_1dx_2\leq c|||\partial_{x_2}\varphi^\varepsilon|||_{L^p(R^\varepsilon)}^p\leq C.
		\end{align*}
		
		Thus, there is $\varphi\in W^{1,p}(\omega )$ such that, up to subsequences, 
		$$
		V^\varepsilon\rightharpoonup\varphi\qquad\mbox{s}-L^p(\omega )\mbox{ and w}-W^{1,p}(\omega ).
		$$
		Moreover, from the one-dimensional Poincaré-Wirtinger inequality we get
		\begin{align*}
		\int_0^{\eps g_\varepsilon(x_1,x_2)}|\varphi^\varepsilon(x_1,x_2,x_3)-V^\varepsilon(x_1,x_2)|^pdx_3=\int_0^{\eps g_\varepsilon(x_1,x_2)}\left|\varphi^\varepsilon-\dfrac{1}{\eps g_0}\int_0^{\eps g_0}\varphi^\varepsilon(x_1,x_2,x_3)dx_3\right|^p dx_3\\
		\leq C(\eps)^p \int_0^{\eps g_\varepsilon(x_1,x_2)}|\partial_{x_3} \varphi^\varepsilon(x_1,x_2,x_3)|^p dx_3.
		\end{align*}
		We integrate the inequality mentioned above in $\omega $ and divide it by $1/\varepsilon$  to derive the following
  
		$$
		|||\varphi^\varepsilon-V^\varepsilon|||_{L^p(R^\varepsilon)}\to0.
		$$
		We also get that 
		$$
		|||\varphi^\varepsilon-\varphi|||_{L^p(R^\varepsilon)}\to0
		$$
		and
		$$
		\|T_\varepsilon\varphi^\varepsilon-\varphi||_{L^p(\omega \times Y^*)}\to 0.
		$$
		Since
		\begin{align*}
		\left\|\partial_{y_1}\mathcal{T}_\varepsilon\varphi^\varepsilon\right\|_{L^p(\omega \times Y^*)}\leq C\varepsilon^\alpha ,\qquad \left\|\partial_{y_2}\mathcal{T}_\varepsilon\varphi^\varepsilon\right\|_{L^p(\omega \times Y^*)}\leq C\varepsilon^\beta \\\qquad\qquad\mbox{and }\qquad\left\|\partial_{y_3}\mathcal{T}_\varepsilon\varphi^\varepsilon\right\|_{L^p(\omega \times Y^*)}\leq C\eps \qquad\qquad
		\end{align*}
		we get
		$$
		\unf{\varepsilon}\varphi^\varepsilon\to \varphi\qquad\mbox{strongly in }L^p(\omega ;W^{1,p}(Y^*)).
		$$
	\end{proof}

\section{Oscillations just in one direction, \texorpdfstring{$\alpha=0$}{gamma}}
 If $\alpha=0$, the thin domain only exhibits oscillations in one variable, thereby placing us in the scenario most akin to those studied in two dimensions. Before we state the result of this section we adapt some definitions and notations to this particular case.
 In this case the unfolding operator acts just in the second variable because the thin domains just present one scale of oscillation. 
 
 We consider
  $$W=\left\{(x_1,x_2,y_2,y_3)\in\mathbb{R}^4:
 (x_1,x_2)\in \omega, \quad (y_2,y_3)\in Y^*(x_1)  \right\}.$$
 $$W_\eps=\left\{(x_1,x_2,y_2,y_3)\in\mathbb{R}^4:
 (x_1,x_2)\in \omega_\eps, \quad (y_2,y_3)\in Y^*(x_1)  \right\}.$$
 where $Y^*(x_1)=\left\{(y_2,y_3)\in\mathbb{R}^2:0<y_2<L_2,\,0<y_3<g(x_1,y_2) \right\}.$

 Then, we define the unfolding operator as follows
 \begin{eqnarray*}
              \mathcal{T}_\varepsilon(\varphi)(x_1,x_2,y_2,y_3)=\left\lbrace\begin{array}{ll}
				\varphi\left(x_1,\varepsilon^\beta\left[\frac{x_2}{\varepsilon^\beta}\right]_{L_2}L_2+\varepsilon^\beta y_2,\eps y_3\right),\,&(x_1,x_2,y_2,y_3)\in W_\eps,\\ 
				0, & (x_1,x_2)\in \Lambda_\varepsilon.
			\end{array}\right.
		\end{eqnarray*}
  Note that all the properties outlined in the previous section can be easily adapted to this specific situation. Additionally, since for each fixed value of $x_1$ a two-dimensional thin domain with oscillating boundary similar to those studied in \cite{AM2}, it is possible to pass to the limit by replicating the results presented in that reference. By Proposition \ref{Proposition_unfolding_properties}, one writes \eqref{variationalproblem} as fixed domain problem as follows
			\begin{equation*}
			\begin{gathered}
			\int_{W}\Big(\frac{\partial\mathcal{T}_\varepsilon(u^\varepsilon)}{\partial x_1}\frac{\partial\mathcal{T}_\varepsilon(\phi)}{\partial x_1} + \mathcal{T}_\varepsilon\Big(\frac{\partial u^\varepsilon}{x_2}\Big)\mathcal{T}_\varepsilon\Big(\frac{\partial \phi}{x_2}\Big)+\mathcal{T}_\varepsilon( u^\varepsilon)\phi\Big) dx_1dx_2dy_2dy_3
			+\dfrac{L_2}{\varepsilon}\int_{R_1^\varepsilon} \big(\nabla u^\varepsilon\nabla\phi+u^\varepsilon\phi\big) dx
		\\=\int_{W}\mathcal{T}_\varepsilon (f^\varepsilon) \mathcal{T}_\varepsilon(\phi) dx_1dx_2dy
			+\dfrac{L_2}{\varepsilon}\int_{R^\varepsilon_1}f^\varepsilon\phi dx, \qquad {\forall \phi \in H^1(\omega)}.
			\end{gathered}
			\end{equation*}
   It is then observed that for all terms, the transition to the limit is immediate, except for the term of the derivative in the direction of the oscillations, $x_2$.
However, for each $x_1$
  fixed, two-dimensional thin domains with oscillations on the upper boundary are recovered, the limit of the unfolding of the derivative in $x_2$ is known from \cite{AM2}. 
  Therefore, based on the results obtained in the literature, we can state the following theorem.
  \begin{theorem}[$\alpha=0$]
  Let $u^\varepsilon$ be the solution of \eqref{problem} with $f^\varepsilon\in L^2(R^\varepsilon)$ such that $|||f^\varepsilon|||_{L^2(R^\varepsilon)}\leq c$ with $c$ a positive constant independent of $\varepsilon$. Suppose that there is $f\in L^2(W)$ such that
		\begin{equation}
		\unf{\varepsilon} f^\varepsilon\rightharpoonup f\quad\textrm{in}\quad L^2(W).
		\end{equation}
		Then, there exists $u\in H^1(\omega )$, such that 
\begin{equation*}
		\unf{\varepsilon}u^\varepsilon\to u\quad\textrm{strongly in}\quad L^2(\omega ;H^1(Y^*(x_1)),
\end{equation*}

Moreover, depending on the value of $\beta$ we have
\begin{itemize}
\item{$0<\beta<1$.} There exists $u_1\in L^2\Big(\omega; H^1_{_{\sharp 2}}\big(Y^*(x_1)\big)\Big)$ with ${\displaystyle \frac{\partial u_1}{\partial y_3}=0}$ such that
\begin{align*}
\mathcal{T_\eps}\Big(\frac{\partial u^\eps}{\partial x_2}\Big)\weto
 \frac{\partial u}{\partial x_2} +\frac{\partial u_1}{\partial y_2} \quad  \hbox{weakly in }\; L^2\big(W\big),
\end{align*}
where $u_1$ is the unique function, up to constants, such that 
$$\frac{\partial u_1}{\partial y_2}(x_1,x_2,y_2) = \Big(-1 + \frac{1}{g(x_1,y_2) \langle\frac{1}{g(x_1,\cdot)}\rangle_{(0,L_2)}}\Big) \frac{\partial u}{\partial x_2}(x_1,x_2),
\quad \hbox{a.e. in }\omega\times (0,L_2)$$
and $u$ is the unique solution of the following Neumann problem:
 \begin{equation}
		\left\lbrace
		\begin{gathered} 
		-\frac{1}{|Y^*(x_1)|}\left(|Y^*(x_1)|u_{x_1}\right)_{x_1}-q_2(x_1)u_{x_2x_2}+u=\bar{f}\quad\textrm{in}\quad \omega ,\\
		\left(|Y^*(x_1)|u_{x_1},q_2(x_1)u_{x_2}\right)\cdot \eta=0\quad\textrm{on}\quad\partial \omega ,
		\end{gathered}
		\right.
		\end{equation}
where $\bar{f}=\langle f \rangle_{Y^*(x_1)}$ and  $q_2(x_1)=\dfrac{1}{\langle g(x_1,\cdot)\rangle_{(0,L_2)}\langle 1/{g(x_1,\cdot)}\rangle_{(0,L_2)}, }$

\item{$\beta=1$.} There exists $u_1\in L^2\Big(\omega; H^1_{\sharp 2}\big(Y^*(x_1)\big)\Big)$  such that
\begin{align*}
&\mathcal{T_\eps}\Big(\frac{\partial u^\eps}{\partial x_2}\Big)\weto
 \frac{\partial u}{\partial x_2}  +\frac{\partial u_1}{\partial y_2} \quad  \hbox{weakly in }\; L^2\big(W\big),\\
& \mathcal{T_\eps}\Big(\frac{\partial u^\eps}{\partial x_3}\Big)\weto \frac{\partial u_1}{\partial y_3} \quad  \hbox{weakly in }\; L^2\big(W\big),
\end{align*}
where $u_1= - X(x_1)(y_2, y_3) \frac{\partial u}{\partial x_2}(x_1,x_2)$ and $X(x_1)\in H^1_{\sharp 2}(Y^*(x_1))$  satisfying ${\displaystyle \int_{Y^*(x_1)} X(x_1) \; dy_2 dy_3  = 0}$ is the unique solution of the following problem 
\begin{equation*} 
\int_{Y^*(x_1)} \nabla X \nabla \psi \, dy_2 dy_3=  \int_{Y^*(x_1)} \frac{\partial \psi}{\partial y_2} \, dy_2 dy_3, \quad \forall \psi \in H^1_{\sharp 2}(Y^*(x_1)).
\end{equation*}
Moreover, $u$ is the unique solution of the following Neumann problem:
 \begin{equation}
		\left\{
		\begin{gathered}
		-\frac{1}{|Y^*(x_1)|}\left(|Y^*(x_1)|u_{x_1}\right)_{x_1}-q_2(x_1)u_{x_2x_2}+u=\bar{f}\quad\textrm{in}\quad \omega ,\\
		\left(|Y^*(x_1)|u_{x_1},q_2(x_1)u_{x_2}\right)\cdot \eta=0\quad\textrm{on}\quad\partial \omega ,
		\end{gathered}
		\right.
		\end{equation}
where $\bar{f}=\langle f \rangle_{Y^*(x_1)}$ and  $q_2(x_1)=\frac{1}{|Y^*(x_1)|} \int_{Y^*(x_1)} \Big\{ 1 - \frac{\partial X(x_1)}{\partial y_2}(y_2,y_3) \Big\} dy_2 dy_3$.
\item{$\beta>1$.} $u$ is the unique weak solution of the following Neumann problem 
\begin{equation*} 
\left\{
\begin{gathered}
-\frac{1}{|Y^*(x_1)|}\left(|Y^*(x_1)|u_{x_1}\right)_{x_1}-g_0(x_1) u_{{x_2}{x_2}} + u =  \bar{f},\quad \textrm{in}\quad\omega,\\
(|Y^*(x_1)|u_{x_1},g_0(x_1)u_{x_2})\cdot \eta=0\quad\textrm{on}\quad\partial \omega,
\end{gathered}
\right.
\end{equation*} 
where $\bar{f}=\langle f \rangle_{Y^*(x_1)}$ and $g_0(x_1)=\min_{y_2\in (0,L)}g(x_1,y_2)$.

\end{itemize}

\end{theorem}

\section{Weak and resonant oscillations, \texorpdfstring{$0<\alpha<\beta\leq1$}{weak}}

 Now we present the central convergence result that will allow us to obtain the homogenized limit problem in all the situations of this section. Convergence of partial derivatives is achieved using auxiliary operators defined in appropriate function spaces, which enable us to separate different scales. In our exploration of oscillatory phenomena, it becomes evident that distinct cases emerge depending on the interplay between various orders of oscillation and the altitude order of the domain.  This approach ensures a systematic and schematic representation, contributing to a more thorough understanding of the intricate dynamics inherent in the relationship between oscillation orders and thickness of the domains.

 The first scenario considered in this section occurs if the frequency order of the oscillations is less than or equal to the order of the domain's height.

 \begin{theorem}\label{Theorem_conv_sequences_Laplacian}
Let $\varphi^\varepsilon\in W^{1,p}(R^\varepsilon)$, $1<p<\infty$, with $\left|\left|\left|\varphi^\varepsilon\right|\right|\right|_{W^{1,p}(R^\varepsilon)}$ uniformly bounded. Assuming $0 < \alpha \leq \beta \leq 1$, there exist $\varphi\in W^{1,p}(\omega ),$ $\varphi_1\in L^p(\omega ;W^{1,p}_\#(Y^*))$ 
%with $\partial_{y_2}\varphi_1=\partial_{y_3}\varphi_1=0$, 
$\varphi_2\in L^p\left(\omega\times(0,L_1);W^{1,p}_\#(Y^*(y_1))\right)$ 
%with $\partial_{y_3}\varphi_2=0$ 
and $\varphi_3\in L^p\left(\omega\times(0,L_1)\times(0, L_2);W^{1,p}(Y^*(y_1,y_2))\right)$ such that, up to subsequences,
		\begin{equation}\label{compa}
		\begin{gathered}
		\mathcal{T}_\varepsilon\varphi^\varepsilon \to  \varphi\mbox{ strongly in } L^p\left(\omega ;W^{1,p}(Y^*)\right)\\
		\mathcal{T}_\varepsilon\partial_{x_1} \varphi^\varepsilon\rightharpoonup \partial_{x_1}\varphi+\partial_{y_1}\varphi_1\mbox{ weakly in }L^p(\omega \times Y^*),\\
				\mathcal{T}_\varepsilon\partial_{x_2} \varphi^\varepsilon\rightharpoonup\partial_{x_2} \varphi+\partial_{y_2}\varphi_2\mbox{ weakly in }L^p(\omega \times Y^*) \\
  \mathcal{T}_\varepsilon\partial_{x_3} \varphi^\varepsilon\rightharpoonup \partial_{y_3}\varphi_3\mbox{ weakly in }L^p(\omega \times Y^*).
		\end{gathered}
		\end{equation}

So that:

\par\noindent i) If $\alpha<\beta<1$ then $\partial_{y_2}\varphi_1=\partial_{y_3}\varphi_1=0$, $\partial_{y_3}\varphi_2=0$.

\par\noindent ii) If $\alpha=\beta<1$ then
$\partial_{y_3}\varphi_1=0$ and $\varphi_2=\varphi_1\in L^p(\omega ;W^{1,p}_\#(Y^*))$.

\par\noindent iii) If $\alpha=\beta=1$, then  $\varphi_3=\varphi_2=\varphi_1\in L^p(\omega ;W^{1,p}_\#(Y^*))$.

\par\noindent iv) If $\alpha<\beta=1$, then $\partial_{y_2}\varphi_1=\partial_{y_3}\varphi_1=0$ and  $\varphi_3=\varphi_2\in L^p\left(\omega\times(0,L_1);W^{1,p}_\#(Y^*(y_1))\right)$.

 \end{theorem}
 \begin{proof}
 First convergence of \eqref{compa} was obtained in Proposition \ref{propositionconvergence}. Taking into account the order of the different microscopic scales involved in the problem, we define the operators $Z^i_\eps, i=1,2,3.$ These operators enable us to achieve the convergence of the unfolded gradient. $Z^1_\eps$ is associated with the smallest exponent, in our case $\alpha$, and is defined as follows
$$Z_\varepsilon^1(x_1,x_2,y)=\dfrac{1}{\varepsilon^\alpha }\left(\unf{\varepsilon}\varphi^\varepsilon(x_1,x_2,y)-\dfrac{1}{|Y^*|}\int_{Y^*}\unf{\varepsilon}\varphi^\varepsilon(x_1,x_2,y) dy\right).$$
Notice that $Z_\eps^1(x_1,x_2,\cdot,\cdot,\cdot,)$ has mean value zero in $Y^*$. Moreover, $Z_\varepsilon^1 \in L^p(\omega ;W^{1,p}(Y^*))$.

Now, for the second smallest exponent, which is $\beta$ in our case, we define $Z^2_\eps$ as follows
$$Z_\varepsilon^2(x_1,x_2,y)=\dfrac{1}{\varepsilon^\beta }\left(\unf{\varepsilon}\varphi^\varepsilon(x_1,x_2,y)-\dfrac{1}{|Y^*(y_1)|}\int_{Y^*(y_1)}\unf{\varepsilon}\varphi^\varepsilon(x_1,x_2,y) dy_2dy_3\right).$$
Observe that $Z_\eps^2(x_1,x_2,y_1,\cdot,\cdot,)$ has mean value zero in $Y^*(y_1)$. Moreover, $Z_\varepsilon^1 \in L^p(\omega\times(0,L_1) ;W^{1,p}(Y^*(y_1))$.

Finally, associated with the largest exponent, we define $Z^3_\eps$ as follows
$$Z_\varepsilon^3(x_1,x_2,y)=\dfrac{1}{\eps }\left(\unf{\varepsilon}\varphi^\varepsilon(x_1,x_2,y)-\dfrac{1}{|Y^*(y_1,y_2)|}\int_{Y^*(y_1,y_2)}\unf{\varepsilon}\varphi^\varepsilon(x_1,x_2,y) dy_2dy_3\right).$$
Observe that $Z_\eps^3(x_1,x_2,y_1,y_2,\cdot,)$ has mean value zero in $Y^*(y_1, y_2)$. Moreover, $Z_\varepsilon^3 \in L^p(\omega\times(0,L_1)\times(0,L_2);W^{1,p}(Y^*(y_1,y_2))$.

Note that the operators are defined in spaces where the derivatives are always uniformly bounded. In fact, we have
\begin{align*}
\partial_{y_1} Z_\varepsilon^1&=\unf{\varepsilon}\partial_{x_1}\varphi^\varepsilon, &\partial_{y_2} Z_\varepsilon^1&=\eps^{\beta-\alpha}\unf{\varepsilon}\partial_{x_2}\varphi^\varepsilon,& \partial_{y_3} Z_\varepsilon^1&=\eps^{1-\alpha}\unf{\varepsilon}\partial_{x_3}\varphi^\varepsilon.\\
  & &\partial_{y_2} Z_\varepsilon^2&=\unf{\varepsilon}\partial_{x_2}\varphi^\varepsilon,
		&\partial_{y_3} Z_\varepsilon^2&=\eps^{1-\beta}\unf{\varepsilon}\partial_{x_3}\varphi^\varepsilon.\\
   & & & &\partial_{y_3} Z_\varepsilon^3&=\unf{\varepsilon}\partial_{x_3}\varphi^\varepsilon. 
\end{align*}
Next, we proceed to define the following operators:
\begin{align*}
&P_\eps^1(x_1,x_2,y)=Z_\eps^1(x_1,x_2,y)-\frac{\partial\varphi}{\partial x_1}\Big(y_1-\frac{1}{|Y^*|}\int_{Y^*}y_1dy\Big),\\
&P_\eps^2(x_1,x_2,y)=Z_\eps^2(x_1,x_2,y)-\frac{\partial\varphi}{\partial x_2}\Big(y_2-\frac{1}{|Y^*(y_1)|}\int_{Y^*(y_1)}y_2dy_2dy_3\Big),\\
&P_\eps^3(x_1,x_2,y)=Z_\eps^3(x_1,x_2,y).
\end{align*}
Notice that all sequences have averages zero in their respective cells. By the Poincaré-Wirtinger inequality, we have that
\begin{align*}
&\| P_\varepsilon^1\|_{L^{p}(\omega\times Y^{*})}\leq C \|\nabla_{y_1y_2y_3} Z_\varepsilon^1\|_{L^p\left(\omega \times Y^*\right)}\leq c,\\
&\| P_\varepsilon^2\|_{L^{p}(\omega\times Y^{*})}\leq C \|\nabla_{y_2y_3} Z_\varepsilon^2\|_{L^p\left(\omega \times Y^*\right)}\leq c,\\
&\| P_\varepsilon^3\|_{L^{p}(\omega\times Y^{*})}\leq C \|\nabla_{y_3} Z_\varepsilon^3\|_{L^p\left(\omega \times Y^*\right)}\leq c,\\
\end{align*}
Hence, there exist $\varphi_1\in L^p(\omega ;W^{1,p}(Y^*))$ with $\partial_{y_1}\varphi_1=\partial_{y_3}\varphi_1=0$, $\varphi_2\in L^p\left(\omega\times(0,L_1);W^{1,p}(Y^*(y_1))\right)$ with $\partial_{y_3}\varphi_2=0$ and $\varphi_3\in L^p\left(\omega\times(0,L_1)\times(0, L_2);W^{1,p}(Y^*(y_1,y_2))\right)$ such that
\begin{align*}
		&P_\varepsilon^1\rightharpoonup \varphi_1\quad\textrm{weakly in}\quad L^p(\omega ;W^{1,p}(Y^*)),\\
		&P_\varepsilon^2\rightharpoonup \varphi_2\quad\textrm{weakly in}\quad L^p\left(\omega\times(0,L_1);W^{1,p}(Y^*(y_2))\right),\\
        &P_\varepsilon^3\rightharpoonup \varphi_3\quad\textrm{weakly in}\quad L^p\left(\omega\times(0,L_1)\times(0, L_2);W^{1,p}(Y^*(y_1,y_2))\right).
\end{align*}
Consequently, taking into account the definition of $Z^i_\eps$ and $P^i_\eps$ we get the desired convergences
\begin{align*}
		&\frac{\partial P_\varepsilon^1}{\partial y_1}=\mathcal{T}_\varepsilon\partial_{x_1} \varphi^\varepsilon-\partial_{x_1}\varphi \rightharpoonup \partial_{y_1}\varphi_1 \quad\textrm{weakly in}\quad L^p(\omega \times Y^*),\\
		&\frac{\partial P_\varepsilon^1}{\partial y_2}=\mathcal{T}_\varepsilon\partial_{x_2} \varphi^\varepsilon-\partial_{x_2}\varphi\rightharpoonup \partial_{y_2}\varphi_2 \quad\textrm{weakly in}\quad L^p(\omega \times Y^*),\\
&\frac{\partial P_\varepsilon^1}{\partial y_3}=\mathcal{T}_\varepsilon\partial_{x_3} \varphi^\varepsilon\rightharpoonup \partial_{y_3}\varphi_3 \quad\textrm{weakly in}\quad L^p(\omega \times Y^*).
\end{align*}

We will finally proof the periodicity of $\varphi_1$, focusing specifically on its periodicity with respect to $y_1$. The periodicity for other variables and functions, such as $\varphi_2$, can be established analogously. The periodicity of $\varphi_1$ with respect to $y_1$ results from the convergence of 
$$\int_{\omega\times Y^*} \Big(P^1_\eps(x_1,x_2,L_1,y_2,y_3)-P^1_\eps(x_1,x_2,0,y_2,y_3)\Big)\psi\,dx_1dx_2, \quad \forall \psi \in \mathcal{D}(\omega \times Y^*).$$
Utilizing both definitions, the $P_1^\epsilon$ and the unfolding operator, and performing a change of variables in the $x_1$ variable, we obtain by integration by parts
\begin{eqnarray}
		&&\int_{\omega \times Y^*}\unf{\varepsilon}\varphi^\varepsilon\dfrac{\psi(x_1-\varepsilon^\alpha L_1,x_2,y)-\psi(x_1,x_2,y)}{\varepsilon^\alpha} dx_1dx_2 dy\\
		&-&\int_{\omega \times Y^*}\left[\frac{\partial\varphi}{\partial x_1}\Big(y_1+L_1-\frac{1}{|Y^*|}\int_{Y^*}y_1dy\Big)-\frac{\partial\varphi}{\partial x_1}\Big(y_1-\frac{1}{|Y^*|}\int_{Y^*}y_1dy\Big)\right]\psi dx_1dx_2 dy\\
		&\stackrel{\varepsilon\to0}{\to}&-\int_{\omega \times Y^*}L_1\varphi\partial_{x_1}\psi dx_1dx_2-\int_{\omega \times Y^*}L_1\partial_{x_1}\varphi\psi dx_1dx_2dy\\
		&=&-\int_{\omega \times Y^*}L_1\varphi\partial_{x_1}\psi dx_1dx_2dy+\int_{\omega \times Y^*}L_1\varphi\partial_{x_1}\psi dx_1dx_2 dy=0,
		\end{eqnarray}
  %The procedure to prove the $L_2$-periodicity in the variable $y_2$
  %is analogous.

 \end{proof}

	Using the previous theorem, we will now obtain the homogenized limit for the different cases depending on the type of oscillations involved.
	
	\subsection{Resonant and Weak oscillations}
	In this subsection, we assume $0<\alpha<1$ and $\beta=1$.  That is, the oscillations in the $x_2$-direction are of order $\eps$ while the oscillations in the $x_1$-direction are much weaker, that is order $\eps^\alpha$.  In particular
 $$	R^\varepsilon=\left\lbrace (x_1,x_2,x_3)\in \mathbb{R}^3:(x_1,x_2)\in \omega,\quad 0<x_3<\varepsilon g\left(\dfrac{x_1}{\varepsilon^\alpha},\dfrac{x_2}{\varepsilon}\right)\right\rbrace,\,\,\,\,0<\varepsilon\ll 1,\, 0<\alpha<1.
$$

We can show the we state the main result.

	\begin{theorem}
		Let $u^\varepsilon$ be the solution of \eqref{problem} with $f^\varepsilon\in L^2(R^\varepsilon)$ such that $|||f^\varepsilon|||_{L^2(R^\varepsilon)}\leq c$ with $c$ a positive constant independent of $\varepsilon$. Suppose that there is $f\in L^2(\omega \times Y^*)$ such that
		\begin{equation}
		\unf{\varepsilon} f^\varepsilon\rightharpoonup  f\quad\textrm{in}\quad L^2(\omega \times Y^*).
		\end{equation}
		Then, there exist $u\in H^1(\omega )$, $u^1\in L^2(\omega \times H^1_\#(Y^*))$, with $\partial_{y_2}u^1=0$ and $\partial_{y_3}u^1=0$, and $u^2\in L^2\left(\omega\times(0,L_2);H^1_{\#2}(Y^*(y_1))\right)$, such that
		\begin{equation}
		\begin{gathered}
		\unf{\varepsilon}u^\varepsilon\to u\quad\textrm{strongly in}\quad L^2(\omega ;H^1(Y^*)),\\
		\unf{\varepsilon}\partial_{x_1}u^\varepsilon\rightharpoonup \partial_{x_1} u +\partial_{y_1}u^1\quad\textrm{weakly in}\quad L^2(\omega \times Y^*),\\
		\unf{\varepsilon}\partial_{x_2}u^\varepsilon\rightharpoonup\partial_{x_2}u+\partial_{y_2}u^2\quad\textrm{weakly in}\quad L^2(\omega \times Y^*),\\
        \unf{\varepsilon}\partial_{x_3}u^\varepsilon\rightharpoonup \partial_{y_3}u^2\quad\textrm{weakly in}\quad L^2(\omega \times Y^*).
		\end{gathered}
		\end{equation}
		The function $u$ is the weak solution of the following problem
        \begin{equation}\label{resoweak}
		\left\lbrace
		\begin{array}{l}
		-q_1u_{x_1x_1}-q_2u_{x_2x_2}+u=\bar{f}\quad\textrm{in}\quad \omega ,\\
		(q_1u_{x_1},q_2u_{x_2})\cdot \eta=0\quad\textrm{on}\quad\partial \omega ,
		\end{array}
		\right.
		\end{equation}
  where $\bar{f}=\langle f \rangle_{Y^*(x_1)}$ and 
  $$q_1=\dfrac{1}{\langle \hat g\rangle_{(0,L_1)}\langle 1/\hat{g}\rangle_{(0,L_1)}}, \quad q_2=\dfrac{1}{|Y^*|}\int_{Y^*}(1-\partial_{y_2}X) dy,$$ 
  where $$\hat{g}(y_1)=\int_0^{L_2}g(y_1,y_2)dy_2$$ and, for every  fixed $y_1  \in (0,L_1)$, $X \in L^2((0,L_1);H^1_\#(Y^*(y_1))$ is the unique solution of the following problem with $\int_{Y^*(y_1)} X dy_2dy_3=0$
  $$\int_{Y^*(y_1)}\nabla_{y_2y_3}X \nabla_{y_2y_3} \psi dy_2dy_3=\int_{Y^*(y_1)}\partial_{y_2}\psi dy_2dy_3,\quad \forall\psi \in H^1(Y^*(y_1)).$$

  	\end{theorem}

	\begin{proof}
	
%	(a) Limit of bounded sequences
			
			Notice that the uniform bound for the solutions of \eqref{problem} are simple to obtain. Just take $\varphi=u^\varepsilon$ in \eqref{variationalproblem}, perform a H\"older's inequality in the right hand side and multiply the resulting inequality by $\varepsilon^{-1}$. This leads to the uniform bound of $|||u^\varepsilon|||_{H^1(R^\varepsilon)}$.
			
			By Theorem \ref{Theorem_conv_sequences_Laplacian}, there are $u\in H^1(\omega )$, $u^2\in L^2(\omega \times(0,L_2);H^1(Y^*(y_1)))$ and $u^1\in L^2(\omega ;H^1(Y^*))$ with $\partial_{y_1}u^1=\partial_{y_3} u^1=0$ such that
			\begin{equation}\label{limitsu1u2RW}
			\begin{gathered}
			\unf{\varepsilon}u^\varepsilon\to u\quad\textrm{strongly in}\quad L^2(\omega ;H^1(Y^*)),\\
		\unf{\varepsilon}\partial_{x_1}u^\varepsilon\rightharpoonup \partial_{x_1} u +\partial_{y_1}u^1\quad\textrm{weakly in}\quad L^2(\omega \times Y^*),\\
		\unf{\varepsilon}\partial_{x_2}u^\varepsilon\rightharpoonup\partial_{x_2}u+\partial_{y_2}u^2\quad\textrm{weakly in}\quad L^2(\omega \times Y^*),\\
        \unf{\varepsilon}\partial_{x_3}u^\varepsilon\rightharpoonup \partial_{y_3}u^2\quad\textrm{weakly in}\quad L^2(\omega \times Y^*).
			\end{gathered}
			\end{equation}
			
			By Proposition \ref{Proposition_unfolding_properties}, one writes \eqref{variationalproblem} as fixed domain problem as follows
			\begin{equation}\tag{PVU}\label{variational_problem_unfolded_RW}
			\begin{gathered}
			\int_{\omega \times Y^*}\mathcal{T}_\varepsilon(\nabla u^\varepsilon) \mathcal{T}_\varepsilon\nabla\varphi +\mathcal{T}_\varepsilon( u^\varepsilon)\varphi dx_1dx_2dy
			+\dfrac{L_1L_2}{\varepsilon}\int_{R_1^\varepsilon}\nabla u^\varepsilon\nabla\varphi+u^\varepsilon\varphi dx
			\\=\int_{\omega \times Y^*}\mathcal{T}_\varepsilon f^\varepsilon \mathcal{T}_\varepsilon\varphi dx_1dx_2dy
			+\dfrac{L_1L_2}{\varepsilon}\int_{R^\varepsilon_1}f^\varepsilon\varphi dx.
			\end{gathered}
			\end{equation}
			
			Taking $\varphi \in H^1(\omega )$ as test functions, one passes to the limit in the above equation using \eqref{limitsu1u2RW}. Then,
			\begin{equation}
			\begin{gathered}\label{prelimitRW}
			\int_{\omega \times Y^*}\left[(\partial_{x_1}u+\partial_{y_1}u^1,\partial_{x_2}u+\partial_{y_2}u^2)\nabla_{x_1x_2}\varphi +u\varphi\right] dx_1dx_2dy
			=\int_{\omega \times Y^*}f\varphi dx_1dx_2dy.
			\end{gathered}
			\end{equation}

	To identify the limit equation we proceed as follows. 		
			Let $\psi\in C^1_{\#2}([0,L_2]\times [0,g_1])$ (the space of functions that are periodic in $y_2$), $\phi_1\in \mathcal{D}(\omega \times(0,L_1))$, $\phi_2\in \mathcal{D}(\omega )$ and $\varPsi\in H^{1}(0,L_1)$. Define
			\begin{equation}
			\begin{gathered}\label{vi_def}
			V^\varepsilon_1(x_1,x_2,x_3)=\varepsilon\phi_1\left(x_1,x_2,\left\{\frac{x_1}{\varepsilon^\alpha} \right\}_{L_1}\right)\psi\left(\frac{x_2}{\varepsilon},\frac{x_3}{\varepsilon}\right)=\varepsilon\phi_1^{\varepsilon}(x_1,x_2)\psi^\varepsilon(x_2,x_3)\\[10pt]
			\textrm{and}\quad V^\varepsilon_2(x_1,x_2)=\varepsilon^\alpha\phi_2(x_1,x_2)\varPsi\left(\frac{x_1}{\varepsilon^\alpha}\right)=\varepsilon^\alpha\phi_2(x_1,x_2)\varPsi^\varepsilon(x_1).
			\end{gathered}
			\end{equation}
			Notice that 
			\begin{equation}
			\left\lbrace\begin{gathered}\label{Theorem_dif_v01}
			\partial_{x_1} V^\varepsilon_1=\varepsilon\partial_{x_1}\phi_1^{\varepsilon}\psi^\varepsilon+\varepsilon^{1-\alpha}\partial_{y_1}\phi_1^{\varepsilon}\psi^\varepsilon,\\
			\partial_{x_2} V^\varepsilon_1= \varepsilon\partial_{x_2}\phi_1^{\varepsilon}\psi^\varepsilon+\phi_1^{\varepsilon}\partial_{y_2}\psi^\varepsilon ,\\
			\partial_{x_3} V^\varepsilon_1=\phi_1\partial_{y_3}\psi^\varepsilon.
			\end{gathered}
			\right. \quad\textrm{and}\quad \left\lbrace
			\begin{gathered}
			\partial_{x_1}V^\varepsilon_2=\varepsilon^\alpha\partial_{x_1}\phi_2\varPsi_\varepsilon+ \phi_2\partial_{y_1}\varPsi^\varepsilon,\\
			\partial_{x_2}V^\varepsilon_2=\varepsilon^{\alpha}\partial_{x_2}\phi_2\varPsi^\varepsilon.
			\end{gathered}\right.
			\end{equation}
%			\begin{equation}\label{Theorem_dif_v02}
%			
%			\end{equation}
			Thus, we can conclude that
			\begin{equation}\label{convergence_test_fct_RW}
            \begin{gathered}
            \unf{\varepsilon}\partial_{y_1} V^\varepsilon_1 \to 0, \quad \unf{\varepsilon}\partial_{y_2} V^\varepsilon_1 \to \partial_{y_2}\psi \phi_1, \quad \unf{\varepsilon}\partial_{y_3} V^\varepsilon_1 \to \partial_{y_3}\psi \phi_1, \\
            \unf{\varepsilon}\partial_{y_1} V^\varepsilon_2 \to \partial_{y_2}\varPsi \phi_2, \quad \unf{\varepsilon}\partial_{y_2} V^\varepsilon_2 \to 0.
            \end{gathered}
            \end{equation}
			strongly in $L^2(\omega \times Y^*)$.

			Now, take $V^\varepsilon_1$ as a test function in \eqref{variational_problem_unfolded_RW} and pass to the limit and obtain
			\begin{equation}
			\int_{\omega \times Y^*}(\partial_{x_1}u+\partial_{y_1}u^1,\partial_{x_2}u+\partial_{y_2}u^2,\partial_{y_3} u^2)(0,\partial_{y_2}\psi,\partial_{y_3}\psi)\phi_1 dx_1dx_2dy=0,
			\end{equation}
			which is equivalent to
			\begin{equation}\label{quasi_auxiliary_problem_RW}
			\int_{Y^*(y_1)}\big(\partial_{x_2}u+\partial_{y_2}u^2\big)\partial_{y_2}\psi+ \partial_{y_3} u^2\partial_{y_3}\psi dy_2dy_3=0 \quad\textrm{a.e.}\quad (x_1,x_2,y_1)\in \omega \times(0,L_1).
			\end{equation}
			
			Notice that $C^1_{\#2}([0,L_2]\times [0,g_1])$ is dense in $H^1(Y^*(y_1))$. This means that we can rewrite the above equality for any $\psi \in H^1(Y^*(y_1))$ and for a.e. point $(x_1,x_2,y_1)\in \omega \times(0,L_1)$. Moreover, for a.e. point $(x_1,x_2,y_1)\in \omega \times(0,L_1)$ it has a unique solution in $H^1(Y^*(y_1))/\mathbb{R}$.
			
			Recall that the uniquely solvable, for each $y_1\in(0,L_1)$, auxiliary problem is
			\begin{equation}
			\begin{gathered}
			\int_{Y^*(y_1)}\nabla_{y_2y_3}X \nabla{y_2y_3} \psi dy_2dy_3=\int_{Y^*(y_1)}\partial_{y_2}\psi dy_2dy_3,\forall\psi \in H^1(Y^*(y_1))\\
			\textrm{with}\quad \int_{Y^*(y_1)} X dy_2dy_3=0.
			\end{gathered}
			\end{equation}
			Thus, comparing \eqref{quasi_auxiliary_problem_RW} with the auxiliary problem above, leads us to
			\begin{equation}\label{Theorem_expression_u1}
			\nabla_{y_2y_3}u^2=-\partial_{x_2}u\nabla_{y_2y_3} X\quad \textrm{for a.e.}\quad (x_1,x_2,y)\in \omega \times Y^*,
			\end{equation}
			that is $u^2$ and $\partial_{x_2}uX$ are equal up to constants.
			
			Next,  take $V^\varepsilon_2$ as a test function in \eqref{variational_problem_unfolded_RW}. Then,
			\begin{equation}
			\int_{\omega \times Y^*}(\partial_{x_1}u+\partial_{y_1}u^1)\partial_{y_1}\varPsi\phi_2 dy=0.
			\end{equation}
			Since the functions are independent of $y_2$ and $y_3$, we can rewrite the above equation as
			\begin{equation}
			\int_0^{L_1}(\partial_{x_1}u+\partial_{y_1}u^1)\partial_{y_1}\varPsi \left(\int_0^{L_2}g(y_1,y_2)dy_2\right)dy_1=0.
			\end{equation}
			Hence, treating $(x_1,x_2)$ as a parameters in the above equation we have that there exists a function $T$ depending on $(x_1,x_2)$ such that
			$$
			(\partial_{x_1}u(x_1,x_2)+\partial_{y_1}u^1(x_1,x_2,y_2)) \left(\int_0^{L_2}g(y_1,y_2)dy_2\right)=T(x_1,x_2)\quad\textrm{a.e.}\quad y_1\in (0,L_1).
			$$
			Using the fact that $\partial_{y_1}u^1$ is periodic, leads us to
			$$
			0=\dfrac{1}{L_1}\int_0^{L_1}\partial_{y_1}u^1 dy_1=\dfrac{1}{L_1}\int_0^{L_1}\dfrac{T}{\hat{g}(y_1)}-\partial_{x_1} u dy_1=\dfrac{T}{L_1}\int_0^{L_1}\dfrac{1}{\hat{g}(y_1)}dy_1-\partial_{x_1} u=T\langle 1/\hat{g}\rangle_{(0,L_1)}-\partial_{x_1}u,
			$$
			where 
			$$
			\hat{g}(y_1)=\int_0^{L_2}g(y_1,y_2)dy_2.
			$$
			Thus, the expression for $\partial_{y_1}u^1$ is
			\begin{equation}\label{Theorem_expression_u2}
			\partial_{y_1}u^1(x_1,x_2,y_1)=\left(\dfrac{1}{\langle 1/\hat{g}(y_1)\rangle_{(0,L_1)}\hat{g}(y_1)}-1\right)\partial_{x_1}u(x_1,x_2)\quad\textrm{a.e.}\quad(x_1,x_2,y_1)\in \omega \times(0,L_1).
			\end{equation}
			
			Finally, taking into account \eqref{Theorem_expression_u1} and \eqref{Theorem_expression_u2} in \eqref{prelimitRW}, leads us to 
            \begin{equation}
			\begin{gathered} 
			\int_{\omega \times Y^*}\dfrac{1}{\langle 1/\hat{g}(y_1)\rangle_{(0,L_1)}\hat{g}(y_1)}\partial_{x_1}u\partial_{x_1}\varphi+ (1-\partial_{y_2}X)\partial_{x_2}u\partial_{x_2}\varphi \,dx_1dx_2dy +u\varphi
			=\int_{\omega \times Y^*}f\varphi \,dx_1dx_2dy.
			\end{gathered}
			\end{equation}
			This concludes the proof, as we have established that $u$ satisfies the weak formulation of the problem shown in the theorem.

	\end{proof}

	\subsection{Weak oscillations}
	
	In this subsection, we suppose oscillations of order $\varepsilon^\alpha$ and $\varepsilon^\beta$ with $0<\alpha<\beta<1$. Then, the thin domain is
	\begin{equation}\label{TDS_W_W}
	R^\varepsilon=\left\{(x_1,x_2,x_3)\in\mathbb{R}^3:(x_1,x_2)\in \omega ,0<x_3<\varepsilon g\left(\frac{x_1}{\varepsilon^\alpha},\frac{x_2}{\varepsilon^\beta}\right) \right\}
	\end{equation}
	%First, we analyze the case where the order of the oscillations differs, $0<\alpha<\beta<1$.

	\begin{theorem}
		Let $u^\varepsilon$ be the solution of \eqref{problem} with $f^\varepsilon\in L^2(R^\varepsilon)$ such that $|||f^\varepsilon|||_{L^2(R^\varepsilon)}\leq c$ with $c$ a positive constant independent of $\varepsilon$. Suppose that there is $f\in L^2(\omega \times Y^*)$ such that
		\begin{equation}
		\unf{\varepsilon} f^\varepsilon\rightharpoonup f\quad\textrm{in}\quad L^2(\omega \times Y^*).
		\end{equation}
		Then, there exist $u\in H^1(\omega )$, $u^1\in L^2(\omega ;H^1_\#(Y^*))$, $u^2\in L^2(\omega \times (0,L_2);H^1_\#(Y^*(y_2)))$ with $\partial_{y_2}u^1=\partial_{y_3}u^1=\partial_{y_3}u^2=0$
		\begin{equation}
		\begin{gathered}
		\unf{\varepsilon}u^\varepsilon\to u\quad\textrm{strongly in}\quad L^2(\omega ;H^1(Y^*)),\\
		\unf{\varepsilon}\partial_{x_1}u^\varepsilon\rightharpoonup \partial_{x_1} u +\partial_{y_1}u^1\quad\textrm{weakly in}\quad L^2(\omega \times Y^*),\\
		\unf{\varepsilon}\partial_{x_2}u^\varepsilon\rightharpoonup\partial_{x_2}u+\partial_{y_2}u^2\quad\textrm{weakly in}\quad L^2(\omega \times Y^*),\\
        \unf{\varepsilon}\partial_{x_3}u^\varepsilon\rightharpoonup 0\quad\textrm{weakly in}\quad L^2(\omega \times Y^*).
		\end{gathered}
		\end{equation}
		The function $u$ is the weak solution of the following problem
		\begin{equation}
		\left\lbrace
		\begin{array}{l}
		-q_1u_{x_1x_1}-q_2u_{x_2x_2}+u=\bar{f}\quad\textrm{in}\quad \omega ,\\
		(q_1u_{x_1},q_2u_{x_2})\eta=0\quad\textrm{on}\quad\partial \omega ,
		\end{array}
		\right.
		\end{equation}
		where
		$$
		\begin{gathered}
%		A=\left(\begin{array}{cc}
%		q_1 & 0\\
%		0 & q_2
%		\end{array}\right),\quad 
        q_1=\dfrac{1}{\langle \bar g\rangle_{(0,L_1)}\langle 1/\bar{g}\rangle_{(0,L_1)}},\quad \bar{g}(y_1)=\int_0^{L_2}g(y_1,y_2)dy_2,\\
        q_2=\frac{1}{L_1}\int_0^{L_1}\dfrac{1}{\langle g\rangle_{(0,L_1)\times (0,L_2)}\langle 1/g(y_1,\cdot)\rangle_{(0,L_2)}} dy_1\quad\textrm{and}\quad
	\bar{f}=\langle f \rangle_{Y^*(x_1)}.
		\end{gathered}
		$$			
	\end{theorem}
	\begin{proof}
		The uniform bounds follows as performed in the previous subsection. Therefore, by Theorem \ref{Theorem_conv_sequences_Laplacian}
		$\partial_{y_2}u^1=\partial_{y_3}u^1=\partial_{y_3}u^2=0$ such that
		\begin{equation}\label{convsol_WW}
		\begin{gathered}
		\unf{\varepsilon}\partial_{x_1}u^\varepsilon\rightharpoonup \partial_{x_1}u+\partial_{y_1}u^1\quad\textrm{weakly in}\quad L^2(\omega \times Y^*),\\
		\unf{\varepsilon}\partial_{x_2}u^\varepsilon\rightharpoonup \partial_{x_2}u+\partial_{y_2}u^2\quad\textrm{weakly in}\quad L^2(\omega \times Y^*),
		\end{gathered}
		\end{equation}
		Pass to the limit \eqref{variational_problem_unfolded_RW} for test functions depending only on $x_1$ and $x_2$. Then,
		\begin{equation}\label{pre_limit_WW}
		\int_{\omega \times Y^*}(\partial_{x_1}u+\partial_{y_1}u^1,\partial_{x_2}u+\partial_{y_2}u^2)\nabla_{x_1x_2}\varphi+u\varphi dx_1dx_2dy=\int_{\omega \times Y^*}\hat{f}\varphi dx_1dx_2dy,\quad\forall\varphi\in H^1(\omega ).
		\end{equation}
		
		Let $\psi\in H^1(0,L_1)$ and $\phi_1\in \mathcal{D}(\omega )$ and $\phi_2\in \mathcal{D}(\omega \times (0,L_1))$. Let $\varPsi\in H^1_\#(0,L_2)$, the functions that are periodic in $y_2$. Define the sequences
		\begin{equation}
		\begin{gathered}
		V^\varepsilon_1(x_1,x_2)=\varepsilon^\alpha \phi_1(x_1,x_2)\psi(x_1/\varepsilon^\alpha),\\
		V^\varepsilon_2(x_1,x_2)=\varepsilon^\beta \phi_2\left(x_1,x_2,\left\{\frac{x_1}{\varepsilon^\alpha} \right\}_{L_1}\right)\varPsi(x_2/\varepsilon^\beta).
		\end{gathered}
		\end{equation}
		The above sequences satisfy
		\begin{align}\label{convvi_WW}
		\unf{\varepsilon}V^\varepsilon_1\to 0, \quad \unf{\varepsilon}\partial_{x_1}V^\varepsilon_1\to\phi_1\partial_{y_1}\psi, \quad\unf{\varepsilon}\partial_{x_2}V^\varepsilon_1\to 0\\
		\unf{\varepsilon}V^\varepsilon_2\to 0, \quad \unf{\varepsilon}\partial_{x_1}V^\varepsilon_2\to 0, \quad \unf{\varepsilon}\partial_{x_2}V^\varepsilon_2\to \phi_2\partial_{y_2}\psi,
        \end{align}
		strongly in $L^2(\omega \times Y^*)$.
		
		Now, take $V^\varepsilon_1$ as a test function in \eqref{variational_problem_unfolded_RW} and pass to the limit. Therefore,
		\begin{equation}
		\int_{\omega \times Y^*}(\partial_{x_1}u+\partial_{y_1}u^1)\phi_1\partial_{y_1}\psi dx_1dx_2dy=0,
		\end{equation}
		which is equivalent to
		\begin{equation}
		\int_0^{L_1}(\partial_{x_1}u+\partial_{y_1}u^1)\partial_{y_1}\psi\left(\int_0^{L_2}g(y_1,y_2)dy_2\right) dy_1=0.
		\end{equation}
		Treating $(x_1,x_2)$ as parameters, we get, in analogy to the previous subsection that 
		\begin{equation}\label{Theorem_expression_u1_WW}
		\partial_{y_1}u^1(x_1,x_2,y_1)=\left(\frac{1}{\langle 1/\bar{g}\rangle_{(0,L_1)}\bar{g}(y_1)}-1\right)\partial_{x_1}u(x_1,x_2)\quad\textrm{a.e.}\quad (x_1,x_2,y_1)\in \omega \times (0,L_1),
		\end{equation}
		where
		$$
		\bar{g}(y_1)=\int_0^{L_2}g(y_1,y_2)dy_2.
		$$
		
		In order to identify $u^2$, take as a test function in \eqref{variational_problem_unfolded_RW} $V^\varepsilon_2$. Then,
		\begin{equation}
		\int_{\omega \times Y^*}(\partial_{x_2}u+\partial_{y_2}u^2)\phi_2\partial_{y_2}\varPsi dx_1dx_2dy=0,
		\end{equation}
		this is equivalent to
		\begin{equation}
		\int_0^{L_2}(\partial_{x_2}u+\partial_{y_2}u^2)\partial_{y_2}\varPsi g(y_1,y_2) dy_2=0.
		\end{equation}
		
		Treating $(x_1,x_2,y_1)$ as parameters, we get 
		$$
		(\partial_{x_2}u+\partial_{y_2}u^2)g(y_1,y_2)=T(x_1,x_2,y_1) \quad\textrm{a.e.}\quad y_2\in (0,L_2).
		$$
		Since $\partial_{y_2}u^2$ is periodic, we get
		\begin{equation}
		0=\frac{1}{L_2}\int_0^{L_2}\partial_{y_2}u^2 dy_2=\frac{1}{L_2}\int_0^{L_2}\frac{T(x_1,x_2,y_1)}{g(y_1,y_2)}-\partial_{x_2} u dy_2=T(x_1,x_2,y_1)\langle 1/g(y_1,\cdot)\rangle_{(0,L_2)}-\partial_{x_2}u.
		\end{equation}
		Put together the two above equalities and obtain
		\begin{equation}\label{Theorem_expression_u2_WW}
		\partial_{y_2}u^2(x_1,x_2,y_1,y_2)=\left(\frac{1}{\langle 1/g(y_1,\cdot)\rangle_{(0,L_2)}g(y_1,y_2)}-1\right)\partial_{x_2}u(x_1,x_2)
		\end{equation}
		a.e. $(x_1,x_2,y_1,y_2)\in \omega \times (0,L_1)\times(0,L_2)$.
		
		Joining together \eqref{pre_limit_WW}, \eqref{Theorem_expression_u1_WW} and \eqref{Theorem_expression_u2_WW} we have
		\begin{equation}
		\int_{\omega \times Y^*}\left(\frac{\partial_{x_1}u}{\langle 1/\bar{g}(y_1)\rangle_{(0,L_1)}\bar{g}(y_1)},\frac{\partial_{x_2}u}{\langle 1/g(y_1,\cdot)\rangle_{(0,L_2)}g(y_1,y_2)}\right)\nabla_{x_1x_2}\varphi+u\varphi dx_1dx_2dy=\int_{\omega \times Y^*}f\varphi dx_1dx_2dy,
		\end{equation}
		which is equivalent to the weak formulation of the limit problem introduced in the theorem.	
	\end{proof}
 %Now, we show the case where the order of the oscillations coincides.

	\section{Strong oscillations $\beta>1$}
In this section, we examine the behavior of solutions to the Neumann problem when at least one direction of oscillation exhibits highly oscillatory behavior
First, we fix some notations. We divide the thin domain $R^\eps$ into two parts: the oscillating and the non-oscillating.
	$$
	R^\varepsilon_-=\omega\times (0,\varepsilon g_0),\;
	R^\varepsilon_+=R^\varepsilon\backslash \overline{R^\varepsilon_+},\;
	R_-=\omega\times (0,g_0).
	$$

 Now we define a rescaling operator and we recall, without proofs, their basic results.
	
	\begin{definition}
		Let $\varphi$ be a measurable function in $R^\varepsilon_-=\omega\times (0,\varepsilon g_0)$. We define $\Pi_\varepsilon:\mathcal{M}(R^\varepsilon_-)\to \mathcal{M}(R_-)$ as follows
		\begin{equation}
		\Pi_\varepsilon(x_1,x_2,x_3)=\varphi(x_1,x_2,\varepsilon x_3),\quad\forall (x_1,x_2,x_3)\in R_-.
		\end{equation}
	\end{definition}
	
	\begin{proposition}\label{Proprescaling}
		\begin{enumerate}
			\item[(i)] Let $\varphi\in L^1(R^\varepsilon_-)$. Then,
			\begin{equation}
			\int_{R_-}\Pi_\varepsilon\varphi  dx=\frac{1}{\varepsilon}\int_{R^\varepsilon_-}\varphi dx.
			\end{equation}
			\item[(ii)] If $\varphi\in L^p(R^\varepsilon_-)$, $1\leq p\leq\infty$, then
			\begin{equation}
			\|\Pi_\varepsilon\varphi\|_{L^p(R_-)}=|||\varphi|||_{L^p(R^\varepsilon_-)}.
			\end{equation}
			\item[(iii)] For $\varphi\in W^{1,p}(R^\varepsilon_-)$, $1\leq p\leq\infty$, 
			\begin{equation}
			\partial_{x_1}\Pi_\varepsilon\varphi =\Pi_\varepsilon\partial_{x_1}\varphi,\quad \partial_{x_2}\Pi_\varepsilon\varphi =\Pi_\varepsilon\partial_{x_2}\varphi\quad\textrm{and}\quad \partial_{x_3}\Pi_\varepsilon\varphi =\varepsilon\Pi_\varepsilon\partial_{x_3}\varphi.
			\end{equation}
			\item[(iv)] Let $\varphi\in L^p(\omega)$, $1\leq p\leq\infty$, then $\Pi_\varepsilon\varphi=\varphi$.
		\end{enumerate}
	\end{proposition}

We also split the basic cell $Y^*$ in two parts. Then, we define 
\begin{align*}
	&Y^*_+=\left\{(y_1,y_2,y_3)\in\mathbb{R}^3:0<y_1<L_1,\,0<y_2<L_2,\,g_0<y_3<g(y_1,y_2) \right\}\\
    &Y_0^*=(0,L_1)\times (0,L_2)\times (0,g_0]
\end{align*}

 Notice that $Y^*_+$ is the part of the basic cell $Y^*$ that oscillates due to the upper boundary. Moreover,   $Y^*=Y_+^*\cup Y_0^*$.

        Since we have two different scales of oscillation for the variables \(x\) or \(y\), the following notation will be used:
      
	$$
	Y^*_+(y_1)=\left\{(y_2,y_3)\in\mathbb{R}^2:0<y_2<L_2,\,g_0<y_3<g(y_1,y_2) \right\}, \quad y_1 \in [0, L_1].
	$$
 	$$
	Y^*_+(y_2)=\left\{(y_1,y_3)\in\mathbb{R}^2:0<y_1<L_1,\,g_0<y_3<g(y_1,y_2) \right\}, \quad y_2 \in [0, L_2].
	$$	
Throughout this section we denote by  $\mathcal{T^+_\eps}$ the unfolding operator associated to the cell $Y^*_+$.
Moreover, in order to simplify the notation, we denote the restriction of the solution of \eqref{problem} to $R^\eps_+$ and $R^\eps_-$ by $u^\varepsilon_+$ and $u^\varepsilon_-$ respectively.

Now, we prove the key compactness result that allows us to transition to the limit in the case where strong oscillations exist.
\begin{theorem}\label{strong0}
		Let $u^\varepsilon$ be the solution of \eqref{problem} with $f^\varepsilon\in L^2(R^\varepsilon)$ such that $|||f^\varepsilon|||_{L^2(R^\varepsilon)}\leq c$ with $c$ a positive constant independent of $\varepsilon$. Then, there exists $u_2 \in L^2(\omega\times Y^*_+)$ such that $$\unf{\varepsilon}^+(\partial_{x_2}u^\varepsilon_+)\rightharpoonup u_2\quad\hbox{weakly in}\quad L^2(\omega \times Y^*_+).$$ 
  Moreover, $u_2$ satisfies
  \begin{equation}
		\tilde{u}_2=0\quad\textrm{a.e. in}\quad \omega \times (0,L_1)\times (0,L_2)\times (g_0,g_1).
		\end{equation}
  \end{theorem}
  \begin{proof}
Rewriting the variational formulation of \eqref{problem} we have:
		\begin{equation}
		\begin{gathered}
		\int_{R^\varepsilon}\left(\partial_{x_1} u^\varepsilon\partial_{x_1}\varphi +\partial_{x_3} u^\varepsilon\partial_{x_3}\varphi +u^\varepsilon\varphi\right) dx\\+\int_{R^\varepsilon_+}\partial_{x_2} u^\varepsilon_+\partial_{x_2}\varphi dx +\int_{R^\varepsilon_-}\partial_{x_2} u^\varepsilon_-\partial_{x_2}\varphi dx =\int_{R^\varepsilon} f^\varepsilon \varphi dx.
		\end{gathered}
		\end{equation}
		Next, by Propostions \ref{Proposition_unfolding_properties} and \ref{Proprescaling} we get that 
		\begin{equation}
		\begin{gathered}\label{variational_unfolded_RS}
		\frac{1}{L_1L_2}\int_{\omega \times Y^*}\left(\unf{\varepsilon}\partial_{x_1} u^\varepsilon\unf{\varepsilon}\partial_{x_1}\varphi +\unf{\varepsilon}\partial_{x_3} u^\varepsilon\unf{\varepsilon}\partial_{x_3}\varphi +\unf{\varepsilon}u^\varepsilon\unf{\varepsilon}\varphi\right) dx_1dx_2dy\\
		+\frac{1}{\varepsilon}\int_{R^\varepsilon_1}\left(\partial_{x_1} u^\varepsilon\partial_{x_1}\varphi +\partial_{x_3} u^\varepsilon\partial_{x_3}\varphi +u^\varepsilon\varphi\right) dx\\
		\frac{1}{L_1L_2}\int_{\omega \times Y^*_+}\unf{\varepsilon}^+\partial_{x_2} u^\varepsilon_+\unf{\varepsilon}^+\partial_{x_2}\varphi dx_1dx_2dy+\frac{1}{\varepsilon} \int_{R^\varepsilon_{+1}}\partial_{x_2} u^\varepsilon_+\partial_{x_2}\varphi dx\\
		+\int_{R_-}\Pi_\varepsilon\partial_{x_2} u^\varepsilon_-\Pi_\varepsilon\partial_{x_2}\varphi dx=\frac{1}{L_1L_2}\int_{\omega \times Y^*}\unf{\varepsilon}f^\varepsilon\unf{\varepsilon}\varphi dx_1dx_2dy+\frac{1}{\varepsilon}\int_{R^\varepsilon_1} f^\varepsilon \varphi dx.
		\end{gathered}
		\end{equation}

        Now, we consider the following function
		\begin{equation}
			w^\varepsilon(x_1,x_2,x_3)=\varepsilon^\beta \tilde{\phi}_1\left(x_1,x_2,\left\{\frac{x_1}{\varepsilon^\alpha} \right\}_{L_1},\frac{x_3}{\varepsilon}\right)\psi\left(\left\{\frac{x_2}{\varepsilon^\beta} \right\}_{L_2}\right),
		\end{equation}
		where $\phi_1\in \mathcal{D}(\omega \times(0,L_1)\times (g_0,g_1))$, $\rho\in \mathcal{D}(0,L_2)$ and $\psi \in \mathcal{D}(0,L_2)$ is choosed satisfying $\psi'=\rho$. Taking $w^\varepsilon$ as a test function in \eqref{variational_unfolded_RS}. Then
		\begin{equation}
		\int_{\omega \times Y^*_+}u_2\phi_1(x_1,x_2,y_1,y_3)\psi'(y_2) dx_1dx_2dy=0
		\end{equation} 
		which implies
		$$
		\int_0^{L_2}\tilde{u}_2(x_1,x_2,y_1,y_2,y_3)\rho(y_2) dy_2=0\quad \textrm{a.e. in}\quad \omega\times (0,L_1)\times (g_0,g_1) 
		$$
		and this means, due to the choice of $\rho$, that 
		\begin{equation}
		\tilde{u}_2=0\quad\textrm{a.e. in}\quad \omega \times (0,L_1)\times (0,L_2)\times (g_0,g_1).
		\end{equation}
		
  \end{proof}

	\subsection{Resonant and strong oscillations}
	
	In this subsection, we suppose oscillations of order $\varepsilon$, that is $\alpha=1$, and $\varepsilon^\beta$ with $0<1<\beta$.
    \begin{equation}\label{TDS_R_S}
	R^\varepsilon=\left\{(x_1,x_2,x_3)\in\mathbb{R}^3:(x_1,x_2)\in \omega ,0<x_3<\varepsilon g\left(\frac{x_1}{\varepsilon},\frac{x_2}{\varepsilon^\beta}\right) \right\}
	\end{equation}

    First, we present a compactness result in which only the direction of oscillation that matches the oscillation frequency of the domain's height is considered.
	\begin{theorem}\label{Theorem_conv_sequences_Laplacian_resonant_strong}
		Let $\varphi^\varepsilon\in W^{1,p}(R^\varepsilon)$ with $|||\varphi^\varepsilon|||_{W^{1,p}(R^\varepsilon)}$ uniformly bounded. Then, there is $\varphi^1\in L^p\left(\omega ;W^{1,p}_\#(Y^*)\right)$ with 
		$\partial_{y_2}\varphi_1=0$ such that
		\begin{equation}\label{cv03}
		\begin{gathered}
		\unf{\varepsilon}\varphi^\varepsilon\to \varphi\quad\textrm{strongly in}\quad L^p\left(\omega ;W^{1,p}_\#(Y^*)\right),\\
		\unf{\varepsilon}\partial_{x_1}\varphi^\varepsilon\rightharpoonup \partial_{x_1}\varphi+\partial_{y_1}\varphi^1\quad\textrm{weakly in}\quad L^p(\omega \times Y^*),\\
		\unf{\varepsilon}\partial_{x_3}\varphi^\varepsilon\rightharpoonup \partial_{y_3}\varphi^1\quad\textrm{weakly in}\quad L^p(\omega \times Y^*).
		\end{gathered}
		\end{equation}
	\end{theorem}
	\begin{proof}
    The proof is omitted as it is analogous to that of Theorem 4.1, considering only the operator $Z_\varepsilon^1$
  that reflects oscillations in the $x_1$ direction. That is,

		Since $\left|\left|\left|\varphi^\varepsilon\right|\right|\right|_{W^{1,p}(R^\varepsilon)}$ is uniformly bounded, one applies Proposition \ref{propositionconvergence} to obtain $\varphi\in W^{1,p}(\omega )$ such that the first convergence of \eqref{cv03} is satisfied.
		$$Z_\varepsilon^1(x_1,x_2,y)=\dfrac{1}{\varepsilon}\left(\unf{\varepsilon}\varphi^\varepsilon(x_1,x_2,y)-\dfrac{1}{|Y^*|}\int_{Y^*}\unf{\varepsilon}\varphi^\varepsilon(x_1,x_2,y)dy\right).$$
  \end{proof}

	\begin{theorem}
		Let $u^\varepsilon$ be the solution of \eqref{problem} with $f^\varepsilon\in L^2(R^\varepsilon)$ such that $|||f^\varepsilon|||_{L^2(R^\varepsilon)}\leq c$ with $c$ a positive constant independent of $\varepsilon$. Suppose that there is $f\in L^2(\omega \times Y^*)$ such that
		\begin{equation}
		\unf{\varepsilon} f^\varepsilon\rightharpoonup f\quad\textrm{in}\quad L^2(\omega \times Y^*).
		\end{equation}
		Then, there exist $u\in H^1(\omega )$ and $u^1\in L^2(\omega ;H^1_\#(Y^*))$ with $\partial_{y_2}u^1=0$
		\begin{equation}
		\begin{gathered}
		\unf{\varepsilon}u^\varepsilon\to u\quad\textrm{strongly in}\quad L^2(\omega ;H^1(Y^*)),\\
		\unf{\varepsilon}\partial_{x_1}u^\varepsilon\rightharpoonup \partial_{x_1}u +\partial_{y_1}u^1\quad\textrm{weakly in}\quad L^2(\omega \times Y^*),\\
        \unf{\varepsilon}\partial_{x_3}u^\varepsilon\rightharpoonup \partial_{y_3}u^1\quad\textrm{weakly in}\quad L^2(\omega \times Y^*),
		\end{gathered}
		\end{equation}
		The function $u$ satisfies
		\begin{equation}
		\left\lbrace
		\begin{array}{l}
		-q_1u_{x_1x_1}-q_2u_{x_2x_2}+u=\bar{f}\quad\textrm{in}\quad \omega ,\\
		(q_1u_{x_1},q_2u_{x_2}))\eta=0\quad\textrm{on}\quad\partial \omega ,
		\end{array}
		\right.
		\end{equation}
		where
		$$
%		A=\left(\begin{array}{cc}
%		q_1 & 0\\0&q_2
%		\end{array}\right),\quad 
        q_1=\frac{1}{|Y^*|}\int_{Y^*}(1-\partial_{y_1}X)dy,\quad q_2=\frac{g_0}{\langle g\rangle_{(0,L_1)\times (0,L_2)}}\quad\textrm{and}\quad \bar{f}=\langle f \rangle_{Y^*}.
		$$
		and $X$ is the solution of the auxiliary problem
		\begin{equation}
		\begin{gathered}
		\int_{Y^*}\left(\partial_{y_1}X\partial_{y_1}\psi+\partial_{y_3}X\partial_{y_3}\psi \right)dy=\int_{Y^*}\partial_{y_1}\psi dy,\quad\forall\psi\in V(Y^*)\\
		\int_{Y^*}X dy=0, X\in V(Y^*)=\{\varphi\in H^1_\#(Y^*):\partial_{y_2}\varphi=0 \}
		\end{gathered}
		\end{equation}
	\end{theorem}
	\begin{proof}
	The uniform bound for the solutions follows as in the previous subsections.	
  Using basic properties of the unfolding and rescaling operators we can rewrite the variational formulation of \eqref{problem} as \eqref{variational_unfolded_RS}. By Theorem \ref{Theorem_conv_sequences_Laplacian_resonant_strong} and \ref{strong0} we can guarantee that there is $u^1\in L^2(\omega ;H^1_\#(Y^*))$ with $\partial_{y_2} u^1=0$ such that
		\begin{equation}
		\begin{gathered}
		\unf{\varepsilon} u^\varepsilon\to u\quad\textrm{strongly in}\quad L^2(\omega ;H^1(Y^*)),\\
		\unf{\varepsilon}\partial_{x_1} u^\varepsilon\rightharpoonup \partial_{x_1} u+\partial_{y_1} u^1\quad\textrm{weakly in}\quad L^2(\omega \times Y^*),\\
		\unf{\varepsilon}\partial_{x_3} u^\varepsilon\to +\partial_{y_3} u^1\quad\textrm{weakly in}\quad L^2(\omega \times Y^*),\\
		\unf{\varepsilon}^+\partial_{x_2} u^\varepsilon_+\rightharpoonup 0\quad\textrm{weakly in}\quad L^2(\omega \times Y^*_+),\\
		\Pi_\varepsilon u^\varepsilon_-\rightharpoonup u^-\quad\textrm{weakly in}\quad H^1(R_-).
		\end{gathered}
		\end{equation}
		
		For test functions depending only on $x_1,x_2$, we can pass to the limit in \eqref{variational_unfolded_RS}, obtaining
		\begin{equation}\label{prelimit_RS}
		\begin{gathered}
		\frac{1}{L_1L_2}\int_{\omega \times Y^*}\left((\partial_{x_1} u+\partial_{y_1} u^1)\partial_{x_1}\varphi +u\varphi\right) dx_1dx_2dy\\
		+\int_{R_-}\partial_{x_2} u^-\partial_{x_2}\varphi dx=\frac{1}{L_1L_2}\int_{\omega \times Y^*}\hat{f}\varphi dx_1dx_2dy.
		\end{gathered}
		\end{equation} 
		
		Let
		\begin{equation}
				\varphi^\varepsilon(x_1,x_2,x_3)=\varepsilon\phi(x_1,x_2)\psi\left(\frac{x_1}{\varepsilon},\frac{x_3}{\varepsilon}\right)=\varepsilon\phi\psi^\varepsilon,
		\end{equation}
		where $\phi\in\mathcal{D}(\omega )$ and $\psi\in H^1_\#(Y^*)$ with $\partial_{y_2}\psi=0$. Notice that
		$$
		\begin{gathered}
		\partial_{x_1}\varphi^\varepsilon=\varepsilon\partial_{x_1}\phi\psi^\varepsilon+\phi\partial_{y_1}\psi^\varepsilon,\\
		\partial_{x_2}\varphi^\varepsilon=\varepsilon\partial_{x_2}\phi\psi^\varepsilon,\\
		\partial_{x_3}\varphi^\varepsilon=\phi\partial_{y_3}\psi^\varepsilon.
		\end{gathered}
		$$
		%and it satisfies
		%$$
		%\unf{\varepsilon}\nabla\varphi^\varepsilon\to \phi(\partial_{y_1}\psi,0,\partial_{y_3}\psi).
		%$$
		
        Then, taking $\varphi^\varepsilon$ as a test function in \eqref{variational_unfolded_RS} leads us to
		$$
		\int_{\omega \times Y^*}\left[(\partial_{x_1} u+\partial_{y_1}u^1)\phi\partial_{y_1}\psi +\partial_{y_3}u^1\phi\partial_{y_3}\psi\right]  dx_1dx_2dy=0.
		$$
		From this and recalling that $X$ satisfies
		\begin{equation}
		\begin{gathered}
		\int_{Y^*}\left(\partial_{y_1}X\partial_{y_1}\psi+\partial_{y_3}X\partial_{y_3}\psi \right)dy=\int_{Y^*}\partial_{y_1}\psi dy,\quad\forall\psi\in V(Y^*)\\
		\int_{Y^*}X dy=0, X\in V(Y^*)=\{\varphi\in H^1_\#(Y^*):\partial_{y_2}\varphi=0 \}
		\end{gathered}
		\end{equation}
		one can obtain that
		$$
		u^1=-\partial_{x_1}u X\quad\textrm{a.e}\quad \omega \times Y^*.
		$$
		
		Notice that
		\begin{equation}
		||\Pi_\varepsilon u^\varepsilon_--u||_{L^2(R_-)}\leq||| u^\varepsilon-u|||_{L^2(R^\varepsilon)}\to0,
		\end{equation}
		which means $u^-=u$ a.e. in $\omega $.

		Finally, one can rewrite \eqref{prelimit_RS} as
		\begin{equation}
		\begin{gathered}
			\int_{\omega } q_1\partial_{x_1}u\partial_{x_1}+q_2\partial_{x_2}u\partial_{x_2}+\varphi+u\varphi dx_1dx_2=\int_{\omega }\bar{f}\varphi dx_1dx_2,\quad\forall\varphi\in H^1(\omega ),
		\end{gathered}
		\end{equation}
		where
		$$
		 q_1=\frac{1}{|Y^*|}\int_{Y^*}(1-\partial_{y_1}X)dy,\quad q_2=\frac{g_0}{\langle g\rangle_{(0,L_1)\times (0,L_2)}}\quad\textrm{and}\quad \bar{f}=\langle f \rangle_{Y^*}.
		$$
	\end{proof}

	\subsection{Weak and strong oscillations}
	
	In this subsection, we suppose oscillations of order $\varepsilon^\alpha$ and $\varepsilon^\beta$ with $0<\alpha<1<\beta$.
	\begin{equation}\label{TDS_S_W}
	R^\varepsilon=\left\{(x_1,x_2,x_3)\in\mathbb{R}^3:(x_1,x_2)\in \omega ,0<x_3<\varepsilon g\left(\frac{x_1}{\varepsilon^\alpha},\frac{x_2}{\varepsilon^\beta}\right) \right\}.
	\end{equation}
    The approach is very similar to the previous case.
	
	\begin{theorem}\label{Theorem_conv_sequences_Laplacian_strong_weak}
		Let $\varphi^\varepsilon\in W^{1,p}(R^\varepsilon)$ with $|||\varphi^\varepsilon|||_{W^{1,p}(R^\varepsilon)}$ uniformly bounded. Then, there are $\varphi\in W^{1,p}(\omega )$ and $\varphi_1\in L^p\left(\omega ;W^{1,p}_\#(Y^*)\right)$ with 
		$\partial_{y_2}\varphi^1=\partial_{y_3}\varphi^1=0$ such that
		\begin{equation}
		\begin{gathered}
		\unf{\varepsilon}\varphi^\varepsilon\to \varphi\quad\textrm{strongly in}\quad L^2\left(\omega ;H^{1}_\#(Y^*)\right)\\
		\unf{\varepsilon}\partial_{x_1}\varphi^\varepsilon\rightharpoonup \partial_{x_1}\varphi+\partial_{y_1}\varphi^1\quad\textrm{weakly in}\quad L^2(\omega \times Y^*).
		\end{gathered}
		\end{equation}
	\end{theorem}
	\begin{proof}
		The proof is analogous to Theorems \ref{Theorem_conv_sequences_Laplacian} and \ref{Theorem_conv_sequences_Laplacian_resonant_strong}.
	\end{proof}
	\begin{theorem}
			Let $u^\varepsilon$ be the solution of \eqref{problem} with $f^\varepsilon\in L^2(R^\varepsilon)$ such that $|||f^\varepsilon|||_{L^2(R^\varepsilon)}\leq c$ with $c$ a positive constant independent of $\varepsilon$. Suppose that there is $f\in L^2(\omega \times Y^*)$ such that
		\begin{equation}
		\unf{\varepsilon} f^\varepsilon\rightharpoonup f\quad\textrm{in}\quad L^2(\omega \times Y^*).
		\end{equation}
		Then, there exist $u\in H^1(\omega )$ and $u^1\in L^2(\omega ;H^1_\#(Y^*))$ with $\partial_{y_3}u^1=\partial_{y_2}u^1=0$ such that
		\begin{equation}
		\begin{gathered}
		\unf{\varepsilon}u^\varepsilon\to u\quad\textrm{strongly in}\quad L^2\left(\omega ;H^{1}_\#(Y^*)\right)\\
		\unf{\varepsilon}\partial_{x_1}u^\varepsilon\rightharpoonup \partial_{x_1}u+\partial_{y_1}u_1\quad\textrm{weakly in}\quad L^2(\omega \times Y^*).
		\end{gathered}
		\end{equation}
		\begin{equation}
		\left\lbrace
		\begin{array}{l}
		-q_1 u_{x_1x_1}-q_2u_{x_2x_2}+u=\bar{f}\quad\textrm{in}\quad \omega ,\\
		\left(q_1u_{x_1},q_2u_{x_2}\right)\eta=0\quad\textrm{on}\quad\partial \omega ,
		\end{array}
		\right.
		\end{equation}
		where 
		\begin{align*}
%		A=\left(\begin{array}{cc}
%		q&0\\ 0& \frac{g_0}{\langle g\rangle }
%		\end{array}\right) \quad\textrm{with}\quad 
        &q_1=\frac{1}{\langle 1/\bar{g}\rangle_{(0,L_1)} \langle g\rangle_{(0,L_1)\times (0,L_2)}},
        \quad q_2=\frac{g_0}{\langle g\rangle_{(0,L_1)\times (0,L_2)}}\\
        &\bar{g}(y_1)=\int_0^{L_2}g(y_1,y_2) dy_2\quad\textrm{and}\quad
		\bar{f}=\langle f \rangle_{Y^*}.
		\end{align*}
		
	\end{theorem}
	\begin{proof}
		Since the solutions of \eqref{problem} are uniformly bounded, we are in position to apply Theorem \ref{Theorem_conv_sequences_Laplacian_strong_weak}, which means that there are $u\in H^1(\omega )$, $u_1\in L^2\left(\omega ;H^1_\#(Y^*)\right)$ and $u_2\in L^2(\omega \times Y^*)$ such that 
		\begin{equation}
		\begin{gathered}
		\unf{\varepsilon}u^\varepsilon\to u\quad\textrm{strongly in}\quad L^2\left(\omega ;H^1_\#(Y^*)\right)\\
		\unf{\varepsilon}\partial_{x_1}u^\varepsilon\rightharpoonup \partial_{x_1}u+\partial_{y_1}u^1\quad\textrm{weakly in}\quad L^2(\omega \times Y^*).
		\end{gathered}
		\end{equation}
		
		By Propositions \ref{Proposition_unfolding_properties} and \ref{Proprescaling} we get that 
		\begin{equation}
		\begin{gathered}\label{variational_unfolded_WS}
		\frac{1}{L_1L_2}\int_{\omega \times Y^*}\left(\unf{\varepsilon}\partial_{x_1} u^\varepsilon\unf{\varepsilon}\partial_{x_1}\varphi +\unf{\varepsilon}\partial_{x_3} u^\varepsilon\unf{\varepsilon}\partial_{x_3}\varphi +\unf{\varepsilon}u^\varepsilon\unf{\varepsilon}\varphi\right) dx_1dx_2dy\\
		+\frac{1}{\varepsilon}\int_{R^\varepsilon_1}\left(\partial_{x_1} u^\varepsilon\partial_{x_1}\varphi +\partial_{x_3} u^\varepsilon\partial_{x_3}\varphi +u^\varepsilon\varphi\right) dx\\
		\frac{1}{L_1L_2}\int_{\omega \times Y^*_+}\unf{\varepsilon}^+\partial_{x_2} u^\varepsilon_+\unf{\varepsilon}^+\partial_{x_2}\varphi dx_1dx_2dy+\frac{1}{\varepsilon} \int_{R^\varepsilon_{+1}}\partial_{x_2} u^\varepsilon_+\partial_{x_2}\varphi dx\\
		+\int_{R_-}\Pi_\varepsilon\partial_{x_2} u^\varepsilon_-\Pi_\varepsilon\partial_{x_2}\varphi dx=\frac{1}{L_1L_2}\int_{\omega \times Y^*}\unf{\varepsilon}f^\varepsilon\unf{\varepsilon}\varphi dx_1dx_2dy+\frac{1}{\varepsilon}\int_{R^\varepsilon_1} f^\varepsilon \varphi dx.
		\end{gathered}
		\end{equation}
		We can pass to the limit the equation above taking test functions depending only on $(x_1,x_2)$ to find
		\begin{equation}
		\begin{gathered}\label{prelimit_WS}
		\frac{1}{L_1L_2}\int_{\omega \times Y^*}\left((\partial_{x_1}u+\partial_{y_1}u_1)\partial_{x_1}\varphi +u\varphi\right) dx_1dx_2dy\\
		+\int_{R_-}\partial_{x_2} u_-\partial_{x_2}\varphi dx=\frac{1}{L_1L_2}\int_{\omega \times Y^*}\hat{f}\varphi dx_1dx_2dy.
		\end{gathered}
		\end{equation}
		To identify $u_1$, one can perform similar arguments as in previous sections: 
		\begin{equation}
		\partial_{y_1}u_1(x_1,x_2,y_1)=\left(\frac{1}{\langle 1/\bar{g} \rangle_{(0,L_1)} \bar{g}(y_1) }-1\right)\partial_{x_1} u\quad\textrm{a.e. in}\quad \omega \times (0,L_1),
		\end{equation}
		where
		$$
		\bar{g}(y_1)=\int_0^{L_2}g(y_1,y_2) dy_2.
		$$
		
		This means that the resulting limit problem is
		\begin{equation}
		\begin{gathered}
			\int_{\omega } q_1\partial_{x_1}u\partial_{x_1}+q_2\partial_{x_2}u\partial_{x_2}+\varphi+u\varphi dx_1dx_2=\int_{\omega }\bar{f}\varphi dx_1dx_2,\quad\forall\varphi\in H^1(\omega ),
		\end{gathered}
		\end{equation}
		where
		\begin{align*}
%		A=\left(\begin{array}{cc}
%		q&0\\ 0& \frac{g_0}{\langle g\rangle }
%		\end{array}\right) \quad\textrm{with}\quad 
        &q_1=\frac{1}{\langle 1/\bar{g}\rangle_{(0,L_1)} \langle g\rangle_{(0,L_1)\times (0,L_2)}},
        \quad q_2=\frac{g_0}{\langle g\rangle_{(0,L_1)\times (0,L_2)}}\\
        &\bar{g}(y_1)=\int_0^{L_2}g(y_1,y_2) dy_2\quad\textrm{and}\quad
		\bar{f}=\langle f \rangle_{Y^*}.
		\end{align*}
	\end{proof}
	
	\subsection{Strong oscillations}
	
	In this subsection, we suppose oscillations of order $\varepsilon^\alpha$ and $\varepsilon^\beta$ with $0<1<\alpha<\beta$.
	\begin{equation}\label{TDS_S_S}
	R^\varepsilon=\left\{(x_1,x_2,x_3)\in\mathbb{R}^3:(x_1,x_2)\in \omega ,0<x_3<\varepsilon g\left(\frac{x_1}{\varepsilon^\alpha},\frac{x_2}{\varepsilon^\beta}\right) \right\}
	\end{equation}
	
	\begin{theorem}
		Let $u^\varepsilon$ be the solution of \eqref{problem} with $f^\varepsilon\in L^2(R^\varepsilon)$ such that $|||f^\varepsilon|||_{L^2(R^\varepsilon)}\leq c$ with $c$ a positive constant independent of $\varepsilon$. Suppose that there is $f\in L^2(\omega )$ satifying		
		\begin{equation}
		\frac{1}{\varepsilon}\int_0^{\varepsilon g\left(\frac{x_1}{\varepsilon^\alpha},\frac{x_2}{\varepsilon^\beta}\right)}f^\varepsilon(x_1,x_2,x_3)dx_3\rightharpoonup f\quad\textrm{in}\quad L^2(\omega ).
		\end{equation}
		Then, there is $u\in H^1(\omega )$ such that
		$$
		\unf{\varepsilon}u^\varepsilon\to u\quad\textrm{in}\quad L^2(\omega ;H^1(Y^*))
		$$
		and 
		\begin{equation}
		\left\lbrace
		\begin{array}{l}
		-q\Delta_{x_1x_2}u+u=\bar f\quad\textrm{in}\quad \omega ,\\
		\frac{\partial u}{\partial \eta}=0\quad\textrm{on}\quad \partial \omega 
		\end{array}
		\right.
		\end{equation}
		where
		$$
		q= \frac{g_0}{\langle g\rangle_{(0,L_1)\times (0,L_2)} }\quad  \hbox{ and } \quad \bar f = \frac{f}{\langle g\rangle}_{(0,L_1)\times (0,L_2)}. 
		$$
	\end{theorem}
	\begin{proof}
		Due to the uniform bounds, we have that there are $u_m\in L^2(\omega \times Y^*_+),u\in H^1(\omega )$, $m=1,2$, such that 
		\begin{equation}
		\begin{gathered}
		\unf{\varepsilon}u^\varepsilon\to u\quad\textrm{strongly in}\quad L^2(\omega ;H^1(Y^*))\\
		\unf{\varepsilon}\partial_{x_1}u^\varepsilon \rightharpoonup u_1\quad\textrm{weakly in}\quad L^2(\omega \times Y^*_+)\\
		\unf{\varepsilon}\partial_{x_2}u^\varepsilon \rightharpoonup u_2\quad\textrm{weakly in}\quad  L^2(\omega \times Y^*_+)\\
		\Pi_\varepsilon\partial_{x_1}u^\varepsilon\rightharpoonup\partial_{x_1} u_-\quad\textrm{weakly in}\quad H^1(R_-)\\
        \Pi_\varepsilon\partial_{x_2}u^\varepsilon\rightharpoonup\partial_{x_2} u_-\quad\textrm{weakly in}\quad H^1(R_-)\\
		\Pi_\varepsilon u^\varepsilon\to u_-\quad\textrm{strongly in}\quad L^2(R_-).
		\end{gathered}
		\end{equation}
		Also, it is simple to see that
		$$
		u_-=u\quad\textrm{a.e. in}\quad \omega .
		$$
		We rewrite the variational formulation of \eqref{problem} with test functions $\varphi\in H^1(\omega )$ as 
		\begin{equation}\label{variational_unfolded_SS}
		\begin{gathered}
			\int_{R_-}\Pi_\varepsilon\nabla_{x_1x_2}u^\varepsilon\Pi_\varepsilon\nabla_{x_1x_2}\varphi dx+\frac{1}{L}\int_{\omega \times Y^*_+}\unf{\varepsilon}^+\nabla_{x_1x_2}u^\varepsilon\unf{\varepsilon}^+\nabla_{x_1x_2}\varphi dx_1dx_2dy\\+\frac{1}{L}\int_{\omega \times Y^*}\unf{\varepsilon}u^\varepsilon\unf{\varepsilon}\varphi dx_1dx_2dy=\frac{1}{\varepsilon}\int_{R^\varepsilon}f^\varepsilon\varphi dx.
		\end{gathered}
		\end{equation}
		We pass to the limit the equation above using Theorem \ref{strong0} and get
		\begin{equation}\label{PrelimitSS}
		\begin{gathered}
			\int_{R_-}\nabla_{x_1x_2}u\nabla_{x_1x_2}\varphi dx+\frac{1}{L_1L_2}\int_{\omega \times Y^*_+}u_1\varphi_x dx_1dx_2dy\\+\frac{1}{L_1L_2}\int_{\omega \times Y^*}u\varphi dx_1dx_2dy=\int_{\omega }f\varphi dx_1dx_2.
		\end{gathered}
		\end{equation}
		
		To prove that  $u_1=0$, we prove first that it independs on the $y_2$ variable. Indeed,
		\begin{equation}
		\begin{gathered}
		\int_{\omega \times Y^*_+}u_1\phi(x_1,x_2)\partial_{2}\psi(y)dx_1dx_2dy
		=\lim_{\varepsilon\to0}\int_{\omega \times Y^*_+}\unf{\varepsilon}^+\partial_{x_1}u^\varepsilon \phi \partial_{y_2}\psi\,dx_1dx_2dy \\
		=\lim_{\varepsilon\to0}\int_{\omega \times Y^*_+}\varepsilon^{-\alpha}\partial_{y_2}\unf{\varepsilon}^+u^\varepsilon \phi\partial_{y_1}\psi\,dx_1dx_2dy= 
		\lim_{\varepsilon\to0}\int_{\omega \times Y^*_+}\varepsilon^{\beta-\alpha}\unf{\varepsilon}^+\partial_{x_2}u^\varepsilon \phi\partial_{y_1}\psi\,dx_1dx_2dy=0.
		\end{gathered}
		\end{equation}
		Next, we follow the steps of the previous sections and define
		$$
		\varphi^\varepsilon(x_1,x_2,x_3)=\varepsilon^\alpha \Tilde{\phi}\left(x_1,x_2,\frac{x_3}{\varepsilon}\right)\psi\left(\left\{\frac{x_1}{\varepsilon^\alpha} \right\}_{L_1}\right) \quad\textrm{in}\quad R^\varepsilon.
		$$
		With this test function in hands, one can prove that
		$$
		\tilde{u}_1=0\quad\textrm{a.e. in}\quad \omega \times (0,L_1)\times (g_0,g_1).
		$$
		Therefore, we rewrite \eqref{PrelimitSS} as follows
		\begin{equation}
		\begin{gathered}
		\int_{\omega }g_0\nabla_{x_1x_2}u\nabla_{x_1x_2}\varphi dx_1dx_2+\frac{|Y^*|}{L_1L_2}\int_{\omega }u\varphi dx_1dx_2=\int_{\omega }f\varphi dx_1dx_2,
		\end{gathered}
		\end{equation}
		obtaining at the effective problem
		\begin{equation}
		\left\lbrace
		\begin{array}{l}
		-\frac{g_0}{\langle g\rangle_{(0,L_1)\times (0,L_2)} }\Delta_{x_1x_2}u+u=\bar f\quad\textrm{in}\quad \omega ,\\
		\frac{\partial u}{\partial \eta}=0\quad\textrm{on}\quad \partial \omega 
		\end{array}
		\right.
		\end{equation}
		with
		$$
		\bar f = \frac{f}{\langle g\rangle}_{(0,L_1)\times (0,L_2)}. 
		$$
	\end{proof}
\ \\
\ \\
\noindent\textbf{Acknowledgments:}\textit{ Jos\'e M. Arrieta is partially supported by grants PID2019-103860GB-I00, PID2022-137074NB-I00  and CEX2023-001347-S ``Severo Ochoa Programme for Centres of Excellence in R\&D'' MCIN/AEI/10.13039/501100011033, the three of them from Ministerio de Ciencia e Innovaci\'on, Spain. Also by ``Grupo de Investigación 920894 - CADEDIF'', UCM, Spain. \\
Jean Carlos Nakasato is partially supported by FAPESP 2023/03847-6,  Brazil. 
\\ Manuel Villanueva-Pesqueira is partially supported by grants PID2019-103860GB-I00, PID2022-137074NB-I00 from MICINN, Spain. Also by ``Grupo de Investigación 920894 - CADEDIF'', UCM, Spain. 
}
\ \\
	
\end{document}